\documentclass[11pt, oneside]{amsart}

\usepackage{graphics,color,pgf,comment}
\usepackage{epsfig}

 \usepackage[ansinew]{inputenc}
 \usepackage[all]{xy}
 \usepackage{hyperref}
\newdir{ >}{!/8pt/@{}*@{>}}

\usepackage{dsfont}
\usepackage{tabularx, hyperref}
\usepackage{caption}
\usepackage{enumerate}

\newcommand{\todo}[1]{\vspace{5mm}\par\noindent
\framebox{\begin{minipage}[c]{0.95 \textwidth} \tt #1
\end{minipage}} \vspace{5mm} \par}

\newtheorem{lemma}{Lemma}[section]
\newtheorem{thm}[lemma]{Theorem}
\newtheorem{prop}[lemma]{Proposition}
\newtheorem{cor}[lemma]{Corollary}

\newtheorem*{prop*}{Proposition}
\newtheorem{prop_intro}{Proposition}
\newtheorem{thm_intro}[prop_intro]{Theorem}

\newtheorem{cor_intro}[prop_intro]{Corollary}
\theoremstyle{definition}
\newtheorem{defn}[lemma]{Definition}

\newtheorem{rem}[lemma]{Remark}

\theoremstyle{definition}

\definecolor{darkgreen}{cmyk}{1,0,1,.2}

 \newcommand\norm{\bBigg@{0.8}}

 \newcommand{\indnorm}[2][flex]{\csname #1l\endcsname\|#2%
                                 \csname #1r\endcsname\|\mathclose{}}
                                  \newcommand{\indnorml}[4][flex]{\csname #1l\endcsname\|#2%
                                 \csname #1r\endcsname\|_{#3}^{#4}\mathclose{}}
\newcommand{\sv}[2][flex]{\indnorm[#1]{#2}}
\newcommand{\isv}[2][norm]{\indnorml[#1]{#2}{\mathcal{I}}{}}

\DeclareMathOperator{\inte}{int}

\DeclareMathOperator{\vol}{vol}
\DeclareMathOperator{\str}{str}
\DeclareMathOperator{\algvol}{algvol}

\newcommand{\calC} {\ensuremath {\mathcal{C}}}

\newcommand{\calS} {\ensuremath {\mathcal{S}}}
\newcommand{\calM} {\ensuremath {\mathcal{M}}}
\newcommand{\calH} {\ensuremath {\mathcal{H}}}

\newcommand{\calU}{\ensuremath {\mathcal{U}}}

\newcommand{\strtil}{\widetilde{\str}}

\newcommand{\R} {\ensuremath {\mathbb{R}}}

\newcommand{\calR} {\ensuremath {\mathcal{R}}}
\newcommand{\bb}{\partial}

\newcommand{\Isom}{\ensuremath{{\rm Isom}}}

\newcommand{\G}{\ensuremath {\Gamma}}

\newcommand{\matH} {\ensuremath {\mathbb{H}}}

\begin{document}

\title[Ideal simplicial volume of manifolds with boundary]{Ideal simplicial volume \\ of manifolds with boundary}

\author[]{R. Frigerio}
\address{Dipartimento di Matematica, Universit\`a di Pisa, Largo B. Pontecorvo 5, 56127 Pisa, Italy}
\email{roberto.frigerio@unipi.it}

\author[]{M. Moraschini}
\address{Dipartimento di Matematica, Universit\`a di Pisa, Largo B. Pontecorvo 5, 56127 Pisa, Italy}
\email{moraschini@mail.dm.unipi.it}

\thanks{}


\begin{abstract}
We define the ideal simplicial volume for compact manifolds with boundary. 
Roughly speaking, the ideal simplicial volume of a manifold $M$
measures the minimal size of possibly ideal triangulations of $M$ ``with real coefficients'', thus providing a variation of the ordinary simplicial volume defined by Gromov in 1982, the main difference being
that ideal simplices are now allowed 
to appear in representatives of the fundamental class.

We show that the ideal simplicial volume is bounded above by the ordinary simplicial volume, and that it vanishes if and only if the ordinary simplicial volume does. We show that, for manifolds with amenable boundary, 
the ideal simplicial volume coincides with the classical one, whereas for hyperbolic manifolds with geodesic boundary it can be strictly smaller. We compute the ideal simplicial volume of an infinite family of hyperbolic $3$-manifolds
with geodesic boundary, for which the exact value of the classical simplicial volume is not known, and we exhibit examples where the ideal simplicial volume provides sharper
bounds on mapping degrees than the  classical simplicial volume. 

%In order to compute the ideal simplicial volume of hyperbolic $3$-man\-ifolds with geodesic boundary we establish a result on volumes of hyperbolic truncated tetrahedra which may be of independent interest.
\end{abstract}

\maketitle

\section*{Introduction}
The simplicial volume  
 is an invariant of manifolds introduced by Gromov in his seminal paper~\cite{Grom82}. If $M$ is a connected, compact, oriented manifold with
(possibly empty) boundary, then the simplicial volume of $M$ is the 
infimum of the sums of the absolute values of the coefficients over all singular chains representing the real fundamental cycle of $M$ (see Subsection~\ref{seminorms:sub}).
We will denote the simplicial volume of $M$ by the symbol $\sv{M}$, both  when $M$ is closed and when $M$ has non-empty boundary. 

If $\tau$ is a triangulation of $M$, then a suitable algebraic sum of the simplices of $\tau$ provides a fundamental cycle for $M$. Hence
the simplicial volume is bounded above by the minimal number of simplices in any triangulation of $M$, and it can be thought as the minimal size of triangulations of $M$ ``with real coefficients''
(even if this analogy is very loose: fundamental cycles of  manifolds can contain simplices which are very far from being embedded, for example). 

When dealing with manifolds with boundary, it is often useful to work with \emph{ideal} triangulations rather than with traditional ones (we refer the reader to Section~\ref{ideal:sec} for the
precise definition of ideal triangulation). Indeed, ideal triangulations are usually more manageable than classical ones, and in many cases they are much more economical: for example, 
the smallest cusped hyperbolic $3$-manifold (which was constructed by Gieseking in 1912, and shown to have the smallest volume among non-compact hyperbolic $3$-manifolds by Adams~\cite{Adams}) can be triangulated using only one ideal simplex,
and the figure-eight knot complement, which doubly covers the Gieseking manifold, admits an ideal triangulation with $2$ ideal tetrahedra. 
%Moreover, in dimension $3$  the ideal triangulations of a manifold $M$ are dual to its \emph{special spines}, which are $2$-subpolyhedra of $M$ on which $M$ deformation retracts, and which are at the base of 
%Matveev's theory of complexity for $3$-manifolds. 

In this paper we introduce and study the notion of \emph{ideal simplicial volume}. To this aim, we first introduce a homology theory for manifolds with boundary, called \emph{marked
homology}, 
in which {ideal} singular simplices are allowed. We then show that marked homology is isomorphic to the singular homology of the manifold relative to its boundary. This
allows us
to define a fundamental class in the marked context, which will be called \emph{ideal fundamental class}. We then define
the ideal simplicial volume $\isv{M}$ of $M$ as the infimum of the $\ell^1$-norms of the representatives of the ideal fundamental class. We refer the reader
to Section~\ref{ideal:sec} for the precise definition. 
%Since the marked homology of a manifold is canonically isomorphic to its relative singular homology,
%one may wonder whether introducing a new homology theory is indeed necessary. 
A crucial fact is that the canonical isomorphism between marked homology and relative homology is not isometric: as discussed in Remark~\ref{no:naif},
a seemingly reasonable definition of the ideal simplicial volume in the context of relative homology would not lead to an interesting theory.

\subsection*{Fundamental properties of the ideal simplicial volume}
%In Section~\ref{ideal:sec} we prove the following basic results about the ideal simplicial volume of manifolds with boundary. 
For every compact manifold with boundary $M$, let $c(M)$ denote the \emph{complexity} of $M$, i.e.~the minimal number of top-dimensional simplices in any ideal triangulation
of $M$.
Just as the simplicial volume of $M$ is bounded from above by the number of top-dimensional simplices in a triangulation of $M$ (see e.g.~\cite[Proposition 1.1]{FFM}), 
the ideal complexity provides an upper bound for the ideal simplicial volume:

\begin{thm_intro}\label{complexity}
Let $M$ be a 
compact manifold with boundary. Then
$$
\isv{M}\leq c(M)\ .
$$
\end{thm_intro}

We also show that the ideal simplicial volume may be exploited to bound the degree of maps between manifolds:

\begin{thm_intro}\label{bound:degree}
 Let $f\colon (M,\bb M)\to (N,\bb N)$ be a map of pairs between compact, connected and oriented manifolds of the same dimension. Then
 $$
 \isv{M}\geq |\deg(f)| \cdot \isv{N}\ .
 $$
\end{thm_intro}

In particular, the ideal simplicial volume is a homotopy invariant of manifolds with boundary (where homotopies are understood to be homotopies of pairs).

The ideal simplicial volume  vanishes if and only if the ordinary simplicial volume does:

\begin{thm_intro}\label{main:inequalities}
There exists a constant $K_n$ only depending on $n\in\mathbb{N}$ such that, for 
 every $n$-dimensional compact manifold $M$, the following inequalities hold: 
 $$
 \isv{M}\leq \sv{M}\leq K_n\cdot \isv{M} .
 $$
 In particular, $\isv{M}=0$ if and only if $\sv{M}=0$.
\end{thm_intro}

\subsection*{Manifolds with amenable boundary}

As Gromov himself pointed out in his seminal paper~\cite{Grom82}, in order to compute the simplicial volume it is often useful to exploit the dual theory of bounded cohomology. 
One of the peculiar features of (singular) bounded cohomology is that it vanishes on spaces with amenable fundamental group. Via some elementary duality results,
this implies in turn that amenable spaces are somewhat invisible when considering their simplicial volume: for example, 
closed manifolds with amenable fundamental group have vanishing simplicial volume, and (under some mild additional hypothesis)
the simplicial volume of manifolds with boundary is additive with respect to gluings along boundary components with amenable fundamental group.
In Section~\ref{amenable:sec} we introduce the dual theory to marked homology  and, building on results from~\cite{BBFIPP,miolibro}, we exploit duality to deduce the following:

\begin{thm_intro}\label{amenable:thm}
Let $M$ be an $n$-dimensional compact manifold, and suppose that the fundamental group of every boundary component of $M$ is amenable. Then
$$
\isv{M}=\sv{M}\ .
$$
\end{thm_intro}

It is proved in~\cite{Lothesis, KK} that, under the assumptions of Theorem~\ref{amenable:thm}, the simplicial volume $\|M\|$ of $M$ coincides with the simplicial volume of the open
manifold $\inte(M)=M\setminus \partial M$ (which is defined in terms of the locally finite homology of $M\setminus \partial M$~\cite{Grom82}), as well as with the Lipschitz simplicial volume
of $\inte(M)$ (see~\cite{Grom82, Loh-Sauer} for the definition). Therefore, for manifolds whose boundary components have amenable fundamental group, 
all these invariants also coincide with the ideal simplicial volume $\isv{M}$ of $M$.

Let $M$ be a complete finite-volume hyperbolic $n$-manifold. As usual, we will denote by $M$ also the natural compactification of $M$, which is a compact manifold
whose boundary components admit a flat Riemannian structure. A celebrated result of Gromov and Thurston shows that
the simplicial volume of $M$ is equal to the ratio $\vol(M)/v_n$, where $\vol(M)$ is the Riemannian volume of $M$ and $v_n$ is the volume of a regular
ideal geodesic simplex in hyperbolic $n$-space $\matH^n$ (all such simplices are isometric to each other). Since flat manifolds
have virtually abelian (hence, amenable) fundamental group, Theorem~\ref{amenable:thm} implies the following:

\begin{cor_intro}\label{hyperbolic:cor}
Let $M$ be (the natural compactification of) a complete finite-volume hyperbolic $n$-manifold. Then
$$
\isv{M}=\sv{M}=\frac{\vol(M)}{v_n}\ .
$$
\end{cor_intro}
In fact, Theorem \ref{amenable:thm} applies to a much bigger class of complete finite-volume manifolds. Let $M$ be an open complete manifold with 
finite volume and pinched negative curvature. As above we still denote by $M$ its natural compactification. 
It is well known that the fundamental group of each boundary component of $M$ is virtually nilpotent, hence amenable
(see e.g.~\cite{BGS85} or \cite[Ex.~19.1]{Bele16}).
Thus
Theorem~\ref{amenable:thm} implies the following:
\begin{cor_intro}\label{neg:pinch:cor}
Let $M$ be (the natural compactification of) an open complete finite-volume negatively pinched $n$-manifold. Then
$$
\isv{M} = \sv{M}.
$$
\end{cor_intro}
%\todo{Non so se la disuguaglianza $\sv{M} \geq C(n, \delta) Vol(M)$ valga solo nel caso chiuso o anche nel caso di volume finito come ad esempio il principio di proporzionalita'. Se valesse si potrebbe mettere nel corollario sopra}

\subsection*{Hyperbolic manifolds with geodesic boundary}
The computation of the simplicial volume of manifolds whose boundary components have \emph{non}-amenable fundamental groups is a very challenging task.
Indeed, the only exact values of the simplicial volume of such manifolds are known for $3$-dimensional handlebodies (and, more in general, for manifolds obtained by attaching
$1$-handles to Seifert manifolds)
and for the product of a surface with the closed interval~\cite{BFP}. In particular, the exact value of the ordinary simplicial volume is not known for any hyperbolic $3$-manifold
with geodesic boundary (and the only available estimates seem quite far from being sharp, see~\cite{BFP2}). On the contrary, 
in Section~\ref{hyperbolic:sec} we compute the exact value of the \emph{ideal}
simplicial volume for an infinite family of hyperbolic $3$-manifolds with geodesic boundary (see Theorem~\ref{geodesic:boundary:thm}
 below). 

In the closed case, 
Gromov's and Thurston's strategy to obtain lower bounds on the classical simplicial volume of hyperbolic manifolds is based on two facts: fundamental cycles may be represented
by linear combinations of geodesic simplices, and the volume of geodesic simplices in hyperbolic space is uniformly bounded.
%Thurston's straightening procedure allows to compute the simplicial volume by restricting to chains supported by geodesic simplices.
When dealing with manifolds with geodesic boundary, the most natural building blocks turn out to be the so-called (partially) truncated simplices.
By mimicking Gromov's and Thurston's strategy, 
%for the computation of a lower bound for the classical simplicial volume of hyperbolic manifolds,
here we exploit upper bounds on the volume
of truncated simplices to obtain lower bounds on the ideal simplicial volume of hyperbolic manifolds with geodesic boundary.

If $M$ is a compact hyperbolic manifold with geodesic boundary, then the \emph{smallest return length} $\ell(M)$ of $M$ is the length of the shortest
path with both endpoints on $\partial M$ which intersects $\partial M$ orthogonally at each of its endpoints. Equivalently, 
it is the smallest distance between distinct boundary components of the universal covering of $M$.

\begin{thm_intro}\label{lowerbound:thm}
Let $M$ be a compact $n$-dimensional hyperbolic manifold with geodesic boundary. Then
$$
\isv{M}\geq \frac{\vol(M)}{V^n_{\ell(M)}}\ ,
$$
where 
$V^n_\ell$ is the supremum of the volumes of $n$-dimensional fully truncated simplices whose edge lengths are not smaller than $\ell$
(see Definition~\ref{defvl}).
\end{thm_intro}

It is well known (see for example \cite{Ush}) that, in dimension 3, the supremum of the volumes of all truncated tetrahedra is given by the volume $v_8\approx 3.664$ 
of the regular ideal octahedron, so Theorem~\ref{lowerbound:thm} implies the following:

\begin{cor_intro}\label{cor:intro:lower}
Let $M$ be a compact hyperbolic $3$-manifold with geodesic boundary. Then
$$
\isv{M}\geq \frac{\vol(M)}{v_8}\ .
$$
\end{cor_intro}

It has been recently proved that, if $\ell$ is sufficiently small, then the constant $V^3_\ell$ coincides with the volume of the regular truncated tetrahedron 
of edge length $\ell$~\cite{FrMo2}.
Together with Theorem~\ref{lowerbound:thm}, this fact allows us 
to compute the exact value of the ideal simplicial volume 
of an infinite family of hyperbolic $3$-manifolds with geodesic boundary.

For every $g\geq 2$ let $\overline\calM_g$ be the set of hyperbolic $3$-manifolds $M$
with connected geodesic boundary such that $\chi(\bb M)=2-2g$ (so $\bb M$, if orientable, is the closed
orientable surface of genus $g$).
Recall that for every $3$-manifold with boundary $M$
the equality $\chi(\bb M)=2\chi (M)$ holds, and in particular $\chi(\bb M)$ is even.
%Since the Euler characteristic of the boundary of any compact $3$-manifold
%is even, 
Therefore, the union $\bigcup_{g\geq 2} \overline\calM_g$ coincides with the set
of hyperbolic $3$-manifolds with connected geodesic boundary.

For every $g\geq 2$ we denote by $\calM_g$ the set of 
3-manifolds $M$ with boundary that admit an ideal triangulation by $g$ tetrahedra
and have Euler characteristic $\chi(M)=1-g$.
Every element of $\calM_g$ has connected boundary 
and supports
a hyperbolic structure with geodesic boundary (which is unique
by Mostow rigidity), hence $\calM_g\subseteq \overline\calM_g$
(see~\cite{FriMaPe}).
Furthermore, Miyamoto proved in~\cite{M}
that elements of $\calM_g$ are exactly the ones having the smallest volume
among the elements of $\overline\calM_g$. In particular, $\calM_g$ is nonempty
for every $g\geq 2$.
The eight elements of $\calM_2$ are exactly the smallest hyperbolic manifolds
with nonempty geodesic boundary~\cite{KM,M}. Some estimates on the ordinary simplicial volume of elements
in $\calM_g$ can be found in \cite{BFP2}. Here we prove the following:

\begin{thm_intro}\label{geodesic:boundary:thm}
Let $M\in\overline{\calM}_g$. Then
$$
\isv{M}\geq g
$$
and
$$
\isv{M}=g
$$
if and only if $M\in\calM_g$. 
\end{thm_intro}

\begin{cor_intro}\label{minimal:isv}
The hyperbolic $3$-manifolds with geodesic boundary having the smallest ideal simplicial volume are exactly the elements of $\calM_2$.
\end{cor_intro}

As a consequence of the previous corollary, among hyperbolic $3$-manifolds with geodesic boundary,
the minimum of the Riemannian volume is attained exactly at those manifolds which  also realize the minimum of the ideal simplicial volume.
%More precisely, among the elements of $\overline{\calM}_g$, the ones having the smallest Riemannian volume are exactly the ones having the smallest ideal simplicial volume.
The same result with ideal simplicial volume replaced by ordinary simplicial volume was conjectured in~\cite{BFP}.

A direct consequence of Theorems~\ref{bound:degree} and~\ref{geodesic:boundary:thm} is the following:

\begin{cor_intro}\label{degree:cor}
Take elements $M\in\calM_g$ and $M'\in\calM_{g'}$, where $g\geq g'$, and let
$$
f\colon (M,\partial M)\to (M',\partial M')
$$
be a map of pairs. Then
$$
\deg(f)\leq \frac{g}{g'}\ .
$$
\end{cor_intro}

As discussed in Subsection~\ref{degrees:sub}, in some cases
 the bound provided by Corollary~\ref{degree:cor} is sharp. Moreover, 
 it is 
 strictly sharper than the bounds 
one can obtain by exploiting the ordinary simplicial volume to study the restriction of $f$ to $\partial M$ or
the extension of $f$ to the double of $M$.
%(see Remark~\ref{} for a thourough discussion of this topic).

\subsection*{Plan of the paper}

In Section \ref{marked:sec} we define  marked spaces and  marked homology. Moreover, we prove some fundamental results about marked homology that are needed for the definition of the ideal simplicial volume. In Section \ref{ideal:sec} we define the ideal simplicial volume and we prove Theorem \ref{complexity} (in Subsection~\ref{isv:vs:c}), 
Theorem \ref{bound:degree} (in Subsection~\ref{bd:deg:sec}) and Theorem~\ref{main:inequalities} (in Subsection~\ref{class:sub}). 
We also introduce marked bounded cohomology, and establish an elementary but fundamental duality result which will be exploited
in the proofs of Theorems~\ref{amenable:thm} and~\ref{lowerbound:thm}.

In Section~\ref{universal:subsec} we introduce the universal covering of marked spaces, while Section \ref{amenable:sec} is devoted to the proof of Theorem \ref{amenable:thm}. Finally, in Section \ref{hyperbolic:sec} we 
focus on hyperbolic manifolds, 
and we prove Theorem \ref{lowerbound:thm}, Theorem \ref{geodesic:boundary:thm} and Corollary \ref{degree:cor}.

\section{Marked homology of marked spaces}\label{marked:sec}
Before defining the ideal simplicial volume of manifolds with boundary we need to introduce and develop the theory of marked
homology for marked spaces. If $M$ is a manifold with boundary, one can consider the quotient space 
$X$ obtained by separately collapsing the connected components of $\partial M$. Such a space $X$ provides the motivating
example for our theory, and the ideal simplicial volume of $M$ will be defined in Section~\ref{ideal:sec} just as the $\ell^1$-seminorm
of the real fundamental class in the marked homology of $X$. This section is mainly devoted to collecting the fundamental properties of marked
homology, and to preparing the ground for the precise definition of the fundamental class in the context of marked homology.

We begin by introducing a  slight generalization of the notion of topological cone.

\begin{defn}\label{coarse:cone}
Let $B$ be a topological space and let $b$ be a point of $B$. We say that $B$ is a \emph{quasicone} with  apex $b$ if the following condition holds: there exists a homotopy
$$
H\colon B\times [0,1]\to B
$$
between the identity and the constant map at $b$ with the following additional properties: 
$H(b,t)=b$ for every $t\in [0,1]$ (i.e.~the homotopy is relative to $\{b\}$), and
for every $x\in B\setminus\{b\}$ the path $H(x,\cdot)\colon [0,1)\to B$ does not pass through $b$.
In other words, we ask that the homotopy $H$ is
such that $H(x,0)=x$ for every $x\in B$, and 
$$H^{-1}(\{b\})=\left(\{b\}\times [0,1]\right)\cup \left(B\times \{1\}\right)\ .$$
\end{defn}

A topological cone is obviously a quasicone, and 
in order to define the ideal simplicial volume of compact manifolds with boundary it would be sufficient to deal with usual cones. Nevertheless, when constructing coverings of marked
spaces it will be useful to work in a slightly more general context.

The following definition singles out the fundamental objects we will be dealing with.
%throughout the whole paper. 

\begin{defn}\label{marked:defn}
A \emph{marked space} $(X,B)$ is a topological pair satisfying the following properties:
\begin{enumerate}
\item
$B$ is closed in $X$.
\item 
For every $b\in B$ there exists a closed neighbourhood
$F_b$ of $b$ in $X$ such that $F_b$ is a quasicone with apex $b$.
Moreover, the $F_b$, $b\in B$, may be chosen to be pairwise disjoint.
\end{enumerate} 
As a consequence of the definition, the subset $B$ is discrete. Moreover, if
$$
F_B=\bigcup_{b\in B} F_b\ ,
$$
then the $F_b$, $b\in B$, are exactly the path connected components of $F_B$.
\end{defn}

\begin{defn}
A map $f\colon (X,B)\to (X',B')$ between marked spaces is \emph{admissible} if and only if  it is continuous and such that $f^{-1}(B')=B$.
\end{defn}

It is immediate to check that there exists a well-defined category having marked spaces as objects and admissible maps as morphisms. 
The notion of homotopic maps admits an obvious admissible version:

\begin{defn}\label{admissible:hom:def}
Let $f,g\colon (X,B)\to (X',B')$ be admissible maps between marked spaces. An admissible homotopy between $f$ and $g$ is an ordinary homotopy
$H\colon X\times [0,1]\to X'$ between $f$ and $g$ such that the map $H(\cdot, t)\colon X\to X'$ is admissible for every $t\in [0,1]$ or, equivalently,
such that $H^{-1}(B')=B\times [0,1]$. Observe that, since $B'$ is discrete, for every $b\in B$ the 
map $t\mapsto H(t,b)$ is constant:
in particular,
$f(b)=g(b)$ for every $b\in B$.
\end{defn}

\label{Ho nuovamente ridato la definizione di spazio marcato senza alcuna ipotesi di compattezza. Speriamo bene.}

\subsection{The marked space associated to a manifold with boundary}\label{associated:marked}
Let $(M, \partial M)$ be an $n$-manifold with boundary.
% (in fact, it would be sufficient to assume that $\partial M$ is compact). 
We associate to $(M,\partial M)$ the marked space $(X,B)$ which is defined as follows:
%If $\partial M$ is compact, then the definition of $(X,B)$ is straightforward:
$X$ is the topological quotient obtained from $M$ by separately collapsing every connected component of $\partial M$, while $B\subseteq X$ is the subset of $X$ given by
the classes associated to the components of $\partial M$. Using that the boundary of $M$ admits a collar in $M$, 
%and the fact that for compact spaces the notions
%of classical and coarse cone coincide, 
it is immediate to check that $(X,B)$ is indeed a marked space in the sense of Definition~\ref{marked:defn}.
We will refer to the pair $(X,B)$ as the marked space associated to $M$, and the
quotient map $p\colon (M,\partial M)\to (X,B)$ will be called the \emph{natural projection}.

\subsection{Marked homology}
As explained in the introduction, the ideal simplicial volume measures the $\ell^1$-seminorm of the ideal fundamental class, which in turn provides a representative of
the fundamental class possibly containing (partially) ideal singular simplices. The following definition gives a precise meaning to the notion of partially ideal singular simplices.

\begin{defn}
Let $(X,B)$ be a marked space. A singular simplex $\sigma\colon \Delta^n\to X$ is \emph{admissible} if $\sigma^{-1}(B)$ is a 
subcomplex of $\Delta^n$, i.e.~a
(possibly empty) union of (not necessarily proper) 
faces of $\Delta^n$. For example, every constant singular simplex is admissible, and any bijective parametrization $\sigma\colon \Delta^n\to X$
of a top-dimensional simplex in an ideal triangulation
of a manifold  with associated marked space $(X,B)$ is admissible, since $\sigma^{-1}(B)$ is equal to the set of vertices of $\Delta^n$.
\end{defn}

Let now $\calR$ be a ring with unity.
It is immediate to check that the restriction of an admissible singular simplex to any of its faces is still admissible (where we understand that each face of 
$\Delta^n$ is identified with $\Delta^{n-1}$ via the unique affine isomorphism which preserves the order of the vertices). This readily implies
that admissible simplices define a subcomplex $\widehat{C}_*^{\calM}(X,B;\calR)$ of the usual singular complex $C_*(X;\calR)$. Of course,
the complex  $\widehat{C}_*^{\calM}(X,B;\calR)$ contains the subcomplex $C_*(B;\calR)$ (which, given that $B$ is discrete, in each degree $n$ only consists of the free $\calR$-module
over the constant $n$-simplices with values in $B$).

We are now ready to introduce the definition of the chain complex that computes the marked homology of marked spaces.

\begin{defn}
Let $(X,B)$ be a marked space. The \emph{marked chain complex} of $(X,B)$ over the coefficient ring $\mathcal{R}$ is the quotient complex
$$
C^\calM_*(X,B;\calR)=
\widehat{C}^\calM_*(X,B;\calR)\big/ C_*(B; \mathcal{R})\ ,
$$
endowed with the differential induced by the usual differential on $C_*(X;\calR)$.

The \emph{marked homology} of $(X,B)$ (with coefficients in $\calR$) is the homology of the marked chain complex $C^\calM_*(X,B;\calR)$, and will be denoted by
$H_{*}^{\mathcal{M}}(X,B; \mathcal{R})$.
\end{defn}

The composition of an admissible singular simplex with an admissible map is still an admissible singular simplex. Using this one can easily check that marked homology
indeed provides a functor on the category of marked spaces. We will see later that it is in fact a homotopy functor, i.e.~that admissibly homotopic maps induce the same morphism
on marked homology.

%\begin{rem}
%In this paper we will consider mainly the cases $\calR=\R$ and $\calR=\matZ$.
%In order to simplify the notation, we will simply denote by $C^\calM_*(X,B)$ and by $H^\calM_*(X,B)$ the marked chain complex and the marked homology
%of $(X,B)$ with \emph{real} coefficients. Moreover, with a slight abuse, when this does not create any ambiguity, we will identify a chain in $C_*^\calM(X,B;\calR)$
%with any of its representatives in $\widehat{C}_*^\calM(X,B;\calR)$.
%\end{rem}

Let $i_*\colon C^\calM_*(X,B;\calR)\to C_*(X,B;\calR)$ be the inclusion of marked chains into ordinary relative chains. The main result of this section is the following:

\begin{thm}\label{main:basic:thm}
For every $n\in\mathbb{N}$, the inclusion $$i_n\colon C^\calM_n(X,B;\calR)\to C_n(X,B;\calR)$$ induces an isomorphism 
$$
H_n(i_n)\colon H^\calM_n(X,B;\calR)\to H_n(X,B;\calR)
$$
between the marked homology and the ordinary relative homology of the pair $(X,B)$ with coefficients in $\calR$.
\end{thm}

This section will be mainly devoted to the proof of
Theorem~\ref{main:basic:thm}, which will be concluded in Subsection~\ref{proof:main:basic:subsec}.

\subsection{The homotopy invariance of marked homology}
Recall that an admissible homotopy between 
admissible maps
$f,g\colon (X,B)\to (X',B')$ is an ordinary homotopy
$H\colon X\times [0,1]\to X'$ between $f$ and $g$ such that $H^{-1}(B')=B\times [0,1]$.

As anticipated above, marked homology is a homotopy functor: 
\begin{thm}\label{homotopy-inv}
Let $f, g \colon (X,B) \rightarrow (X',B')$ be admissibly  homotopic admissible maps.
Then the induced homomorphisms $$H^{\mathcal{M}}_{n}(f), H^{\mathcal{M}}_{n}(g)\colon H^\calM_n(X,B;\calR)\to H^\calM_n(X',B';\calR)$$
coincide for every $n\in\mathbb{N}$.
\end{thm}
\begin{proof}
The proof is a slight modification of the one for ordinary singular homology (see for instance \cite[Thm. 2.10]{hatcher}). 
As in the classical case, we subdivide the prism $\Delta^{n} \times I$ into $(n+1)$-dimensional simplices as follows. 
For every $i=0,\ldots,n$ we set $v_i=(e_i,0)$, $w_i=(e_i,1)$, where $e_i$ is the $i$-th vertex of $\Delta^n$. We then denote
by $\sigma_i\colon \Delta^{n+1}\to \Delta^n\times I$ the affine isomorphism sending the vertices of $\Delta^{n+1}$ to the
vertices  $v_{0}, \cdots, v_{i}, w_{i}, \cdots, w_{n}$ of $\Delta^n\times I$.
%We recall how this subdivision can be performed. First of all we label the vertices of $|\Delta^{n}| \times \{0\}$ and $|\Delta^{n}| \times \{1\}$ with $[v_{0}, \cdots, v_{n}]$ and $[w_{0}, \cdots, w_{n}]$, respectively. Then, the $(n+1)$-simplices are the ones labelled with $[v_{0}, \cdots, v_{i}, w_{i}, \cdots, w_{n}]$, where $i$ varies in $\{0, \cdots, n\}$.

Let $H \colon X \times I \rightarrow X'$ be an admissible homotopy between $f$ and $g$.
Then the usual homotopy operator $T_n\colon C_n(X;\calR)\to C_{n+1}(X';\calR)$ is the unique $\calR$-linear map such that, for every singular simplex $\sigma\colon \Delta^n\to X$,
%Then, we can consider the composition $$F \circ (\sigma \times \mathbb{1}) \colon | \Delta^{n}| \times I \rightarrow X \times I \rightarrow Y.$$ Using the composition, we can define the prism operator $P \colon C_{n}^{\mathcal{M}}(X,B) \rightarrow C_{n+1}^{\mathcal{M}}(Y, B')$ by the following formula: 
$$T_n(\sigma)=\sum_{i = 0}^{n} (-1)^{i} H \circ (\sigma \times \mathds{1}) \circ \sigma_i\ .$$

It is clear that $T_n$ sends $C_n(B;\calR)$ to $C_{n+1}(B';\calR)$, so in order to conclude the proof it is sufficient to show that
$T_n(\sigma)\in \widehat{C}_{n+1}^\calM (X',B;\calR)$ provided that
$\sigma \colon \Delta^{n} \rightarrow X$ is admissible. 

Let us denote by $K_i\subseteq \Delta^n\times I$ the image of the affine embedding
$\sigma_i\colon \Delta^{n+1}\to \Delta^n\times I$, and observe that $\Delta^n\times I$ admits a structure of simplicial complex whose
top-dimensional simplices are exactly the $K_i$. Let $S=\sigma^{-1}(B)\subseteq \Delta^n$. Since $\sigma$ is admissible, the subset $S$
is a subcomplex of $\Delta^n$. From the very definition of admissible homotopy it follows that
$(H\circ (\sigma\times I))^{-1}(B')=S\times I$, so the pair
$$
\left(\Delta^{n+1},\left(H \circ (\sigma \times \mathds{1}) \circ \sigma_i\right)^{-1}(B')\right) 
$$
is affinely isomorphic (via $\sigma_i$) to the pair $(K_i,(S\times I)\cap K_i)$. But $S\times I$ is a subcomplex of $\Delta^n\times I$ (with respect to
the structure of simplicial complex described above), so $(S\times I)\cap K_i$ is a subcomplex of $K_i$, and this implies that
the singular simplex $H \circ (\sigma \times \mathds{1}) \circ \sigma_i\colon \Delta^{n+1}\to X'$ is admissible. We have thus proved
that the homotopy operator $T_*$ induces a well-defined homotopy operator
$$
T'_*\colon C_*^\calM(X,B;\calR)\to C_{*+1}^\calM(X',B';\calR)\ .
$$
Now the conclusion follows from
the very same argument for ordinary singular homology (see e.g.~\cite[Thm. 2.10]{hatcher}).
\end{proof}

An admissible map $f \colon (X,B) \rightarrow (X',B')$ is an \emph{admissible homotopy equivalence} if there exists 
an admissible map $g \colon (X',B') \rightarrow (X,B)$ such that $g \circ f$ and $f \circ g$ are admissibly homotopic to the identity maps $Id_{(X,B)}$ and $Id_{(X',B')}$, respectively.
If this is the case, then $f$ restricts to a bijection between $B$ and $B'$. Moreover, Theorem~\ref{homotopy-inv} immediately implies the following:
\begin{cor}
Let $f \colon (X,B) \rightarrow (X',B')$ be an admissible homotopy equivalence. 
Then the induced map $f_* \colon H_*(X,B;\calR)\to H_*(X',B';\calR)$ is an isomorphism in every degree.
\end{cor}

\subsection{Marked homology of quasicones}
An important ingredient in our proof of Theorem~\ref{main:basic:thm} is the vanishing of marked homology of quasicones.
\begin{prop}\label{PropVanishingHomologyForCones}
Let $F$
be a quasicone with apex $b\in F$. Then
$$
H_n^\calM(F,\{b\};\calR)=0
$$
for every $n\geq 0$ and every coefficient ring $\calR$.
\end{prop}
\begin{proof}
Unfortunately, in order to prove the proposition we are not allowed to apply the homotopy invariance of marked homology
to the contracting homotopy which retracts $F$ onto $b$: indeed, such a homotopy is not admissible, and even the constant map
$F\to F$ sending every point of the cone to $b$ is not admissible. Nevertheless, the obvious algebraic contracting homotopy for singular chains on $F$
takes admissible chains to admissible chains, thus yielding the desired vanishing of marked homology for quasicones. 

Let us now give some details.
Recall that by definition of quasicone there exists a homotopy 
$$
H\colon F\times [0,1]\to F
$$
between the identity and the constant map at $b$ which satisfies 
$$H^{-1}(b) = \left(\{b\} \times I\right) \cup \left( F \times \{1\}\right)\ .$$ 

Let $\sigma_i\colon \Delta^{n+1}\to \Delta^n\times I$, $i=0,\ldots,n$, be the affine parametrizations of the top-dimensional
simplices of the decomposition of $\Delta^n\times I$ described in the proof of Theorem \ref{homotopy-inv}. 
%Recall that the standard construction which shows that the identity and the constant map at $b$ induce the same morphism in marked homology. The key point is to subdivide the prism $|\Delta^{n}| \times I$ into $(n+1)$-simplices through the following subdivision. First of all we label the vertices of $|\Delta^{n}| \times \{0\}$ and $|\Delta^{n}| \times \{1\}$ with $[v_{0}, \cdots, v_{n}]$ and $[w_{0}, \cdots, w_{n}]$, respectively. Then, the $(n+1)$-simplices are the ones labelled with $[v_{0}, \cdots, v_{i}, w_{i}, \cdots, w_{n}]$, where $i$ varies in $\{0, \cdots, n\}$.
We consider as before the usual homotopy operator $T_n\colon C_n(F;\calR)\to C_{n+1}(F;\calR)$  such that, for every singular simplex $\sigma\colon \Delta^n\to F$,
%Then, we can consider the composition $$F \circ (\sigma \times \mathbb{1}) \colon | \Delta^{n}| \times I \rightarrow X \times I \rightarrow Y.$$ Using the composition, we can define the prism operator $P \colon C_{n}^{\mathcal{M}}(X,B) \rightarrow C_{n+1}^{\mathcal{M}}(Y, B')$ by the following formula: 
$$T_n(\sigma)=\sum_{i = 0}^{n} (-1)^{i} H \circ (\sigma \times \mathds{1}) \circ \sigma_i\ .$$

Of course $T_n$ sends $C_n(\{b\};\calR)$ to $C_{n+1}(\{b\};\calR)$, and we need to check that $T_n(\sigma)\in \widehat{C}_{n+1}^\calM (F,\{b\};\calR)$ provided that
$\sigma \colon \Delta^{n} \rightarrow F$ is admissible. 

Let  $K_i\subseteq \Delta^n\times I$ be the image of the affine embedding
$\sigma_i\colon \Delta^{n+1}\to \Delta^n\times I$, and recall that $\Delta^n\times I$ admits a structure of simplicial complex whose
top-dimensional simplices are exactly the $K_i$. Let $S=\sigma^{-1}(B)\subseteq \Delta^n$. Since $\sigma$ is admissible, the subset $S$
is  a subcomplex of $\Delta^n$. From the fact that $H^{-1}(\{b\})=(\{b\}\times I)\cup (F\times \{1\})$ it follows that
$(H\circ (\sigma\times I))^{-1}(\{b\})=(F\times I)\cup (\Delta^n\times \{1\})$, so the pair
$$
\left(\Delta^{n+1},\left(H \circ (\sigma \times \mathds{1}) \circ \sigma_i\right)^{-1}(\{b\})\right) 
$$
is affinely isomorphic (via $\sigma_i$) to the pair $(K_i,((F\times I)\cup (\Delta^n\times \{1\}))\cap K_i)$. This implies that
the singular simplex $H \circ (\sigma \times \mathds{1}) \circ \sigma_i\colon \Delta^{n+1}\to X'$ is admissible. We have thus proved
that the homotopy operator $T_*$ induces a well-defined homotopy operator
$$
T'_*\colon C_*^\calM(F,\{b\};\calR)\to C_{*+1}^\calM(F,\{b\};\calR)\ 
$$
between the identity (since $H(\cdot,0)$ is the identity of $F$) and the zero map (since $H(\cdot,1)$ is the constant map at $b$, and simplices
supported in $\{b\}$ are null in the marked chain complex). This concludes the proof.
\end{proof}

%\begin{cor}\label{isomF}
%Let $j_*\colon C_*^\calM(F_B,B;\calR)\to C_*(F_B,B;\calR)$ be the obvious inclusion. Then the induced map
%$$
%H_n(j_n)\colon H_n^\calM(F_B,B;\calR)\to H_n(F_B,B;\calR)
%$$ 
%is an isomorphism for every $n\in\mathbb{N}$.
%\end{cor}
%\begin{proof}
%The case $n=0$ is easy and left to the reader, while Proposition~\ref{PropVanishingHomologyForCones}
%and the fact that $F_B$ deformation retracts onto $B$ ensure that
% $$H_n^\calM(F_B,B;\calR)= H_n(F_B,B;\calR)=0$$ for every $n\geq 1$.
%\end{proof}

\subsection{Small Marked Homology}
A key fact in the proof of the Excision Property (or of the Mayer-Vietoris sequence) for ordinary singular homology is that
singular homology may be computed by simplices which are supported in the elements of an open cover. In this subsection we show that the same property holds for
marked homology.

\begin{defn}
Let $(X,B)$ be a marked space and let $\mathcal{U}$ be a family of subsets of $X$ such that
$$
X\subseteq \bigcup_{U\in\calU} \inte(U)\ ,
$$
where $\inte(U)$ denotes the biggest open set contained in $U$. 
We say that a(n admissible) singular simplex $\sigma \colon \Delta^{n} \rightarrow X$ is $\calU$-\emph{small} if its image  lies entirely in some $U \in \mathcal{U}$. 
We denote by $C_*^\calU(X,B;\calR)$ (resp.~$C_*^{\calU,\calM}(X,B;\calR)$) the submodule of $C_*(X,B;\calR)$ (resp.~of $C_*^{\calM}(X,B;\calR)$) generated
by the classes of $\calU$-small simplices (resp.~$\calU$-small admissible simplices).

Then $C_*^\calU(X,B;\calR)$ (resp.~$C_*^{\calU,\calM}(X,B;\calR)$) is a subcomplex of $C_*(X,B;\calR)$ (resp.~of $C_*^{\calM}(X,B;\calR)$),
whose homology is denoted by $H_*^\calU(X,B;\calR)$ $\, $(resp. $H_*^{\calU,\calM}(X,B;\calR)$).

%Clearly, the boundary of a small marked singular simplex is again a small marked singular simplex and so there exists a well-defined \emph{small} chain complex associated to $(X,B)$. Let us denote it by $(C_{n}^{\mathcal{U}, \mathcal{M}}(X,B), \partial_{n})$. The homology associated to this chain complex (notice that $\partial^{2} = 0$) is called \emph{small marked homology} and denoted by $H^{\mathcal{U}, \mathcal{M}}(X,B)$.
\end{defn}

The following theorem extends  a fundamental result about  ordinary singular homology to marked homology:

\begin{thm}\label{TeorSmall}
Let $s_{*} \colon C_{n}^{\mathcal{U}, \mathcal{M}}(X,B;\calR) \rightarrow C_{n}^{\mathcal{M}}(X,B;\calR)$ be the obvious inclusion.
%be the chain map induced by the inclusion.
 Then, there exists a chain map $\rho_* \colon C_{*}^{\mathcal{M}}(X,B;\calR) \rightarrow C_{*}^{\mathcal{U}, \mathcal{M}}(X,B;\calR)$ such that 
 the compositions $s_{*} \circ \rho_*$ and $\rho_* \circ s_{*}$ are both chain homotopic to the identity.
\end{thm}

\begin{proof}
The proof for ordinary singular homology (see e.g.~\cite[Prop. 2.21]{hatcher}) works also in the context of marked homology. Here below we focus our attention on
some subtleties that arise in the context of marked homology.
%is based on the same arguments of the ordinary one about singular homology. This is the reason why we will focus our attention only on the slight differences. We refer the reader to \cite[Prop. 2.21]{hatcher}.

Let $S_*\colon C_*(X;\calR)\to C_*(X;\calR)$ be the usual barycentric subdivision operator. We would like to prove that $S_*$ induces a well-defined operator
$S_*^\calM\colon C_*^\calM (X,B;\calR)\to C_*^\calM(X,B;\calR)$ on marked chains. The fact that $S_*$ sends chains supported in $B$ to chains supported in $B$ is obvious,
so we need to show that $S_*(\sigma)$ is an admissible chain provided that the singular simplex $\sigma$ is admissible. To this aim, observe that
a singular simplex $\sigma\colon \Delta^n\to X$ is admissible if and only if the following condition holds: let $p\in \Delta^n$ and let $D(p)\subseteq \Delta^n$ be the unique
open face of $\Delta^n$ containing $p$ (such a face may have any dimension between $0$, when $p$ is a vertex of $\Delta^n$, and
$n$, when $p$ lies in the interior of $\Delta^n$); if $\sigma(p)\in B$, then $\sigma(\overline{D(p)})\subset B$. 

Suppose now that $\sigma\colon \Delta^n\to X$ is admissible, let $K\subseteq \Delta^n$ be a geometric $n$-simplex appearing in the barycentric decomposition of $\Delta^n$,
and let $\tau\colon \Delta^n\to K$ be an affine parametrization of $K$. In order to show that $S_*(\sigma)$ is an admissible chain it is sufficient to prove
that $\sigma\circ \tau$ is an admissible singular simplex. However, let $p\in \Delta^n$ be such that $\sigma(\tau(p))\in B$. Let $D(p)$ (resp.~$D(\tau(p))$) be the unique open 
face of $\Delta^n$ containing $p$ (resp.~$\tau(p)$), and observe that $\tau(\overline{D(p)})\subseteq \overline{D(\tau(p))}$. Thus, if $\sigma(\tau(p))\in B$, then
by admissibility of $\sigma$ we have $\sigma (\overline{D(\tau(p))})\subseteq B$, hence $\sigma(\tau(\overline{D(p)}))\subseteq B$. This proves that
$\sigma\circ\tau$ is admissible, thus showing that the operator $S_*^\calM$ is indeed well defined.

In order to prove the theorem we also need to show that the operator $S_*^\calM$ is homotopic to the identity. To this aim, let
$T_*\colon C_*(X;\calR)\to C_{*+1}(X;\calR)$ be the standard homotopy between $S_*$ and the identity of $C_*(X;\calR)$ (as defined e.g. in~\cite[Prop. 2.21]{hatcher}). 
As usual we need to prove
that $T_*$ preserves admissible chains.
To this aim we first describe the geometric meaning of $T_*$. 
We inductively triangulate the prism $\Delta^{n} \times I$ by taking the cone of the whole triangulated boundary  $\left(\Delta^{n} \times \{0\}\right) \bigcup \left(\partial \Delta^{n} \times I\right)$
with respect to the barycenter of $\Delta^{n} \times \{1\}$. We also fix an arbitrary top-dimensional simplex $K$ of the triangulation of $\Delta^n\times I$ just described, and
an affine parametrization $\tau\colon \Delta^{n+1}\to K$ of $K$. 

Let now $\sigma\colon \Delta^n\to X$ be an admissible simplex, and let $\pi\colon \Delta^{n}\times I\to\Delta^n$ be the projection
onto the first factor. In order to prove that $T_*$ preserves admissible chains, it is sufficient to show that the singular simplex
$$
\sigma \circ \pi\circ \tau\colon \Delta^{n+1}\to X
$$
is admissible. As before, for every $p\in\Delta^{n+1}$ we let $D(p)$ be the smallest open face of $\Delta^{n+1}$ containing $p$. Moreover,
for every $q\in \Delta^n\times I$ we define $C(q)$ as the smallest open cell containing $q$ in the product cell structure of $\Delta^n\times I$
(where we understand that $\Delta^n$ and $I$ are endowed with the cellular structure induced by their simplicial structure). Then it is easy to check
that $\tau(\overline{D(p)})\subseteq \overline{C(\tau(p))}$. As a consequence, if $\sigma(\pi(\tau(p)))\in B$, then by admissibility of $\sigma$
we have $\sigma (\overline{D(\pi(\tau(p)))})\subseteq B$, hence $\sigma (\pi(\overline{C(\tau(p))}))\subseteq B$. But this implies in turn that
$\sigma(\pi(\tau(\overline{D(p)})))\subseteq B$, i.e.~that $\sigma\circ\tau$ is admissible.

We have thus proved that the homotopy operator $T_*$ induces an operator
$T_*^\calM\colon C_*^\calM(X,B;\calR)\to C_{*+1}^\calM(X,B;\calR)$ which realizes a homotopy between $S_*^\calM$ and the identity
of $C_*^\calM(X,B;\calR)$.
Now the conclusion follows from the very same 
 arguments described in~\cite[Prop. 2.21]{hatcher} for ordinary singular homology.
\end{proof}

\begin{cor}\label{TeorSmall:cor}
The inclusion $s_{*} \colon C_{n}^{\mathcal{U}, \mathcal{M}}(X,B;\calR) \rightarrow C_{n}^{\mathcal{M}}(X,B;\calR)$  induces an isomorphism
$$
H_n^\calM(s_n)\colon H_n^{\calU,\calM}(X,B;\calR)\to H_n^\calM(X,B;\calR)
$$
in every degree $n\in\mathbb{N}$.
\end{cor}

We are now ready to provide the proof of Theorem~\ref{main:basic:thm}, which states that
 the inclusion $$i_n\colon C^\calM_n(X,B;\calR)\to C_n(X,B;\calR)$$ induces an isomorphism 
$$
H_n(i_n)\colon H^\calM_n(X,B;\calR)\to H_n(X,B;\calR)
$$
for every $n\in\mathbb{N}$.

\subsection{Proof of Theorem~\ref{main:basic:thm}}\label{proof:main:basic:subsec}
Let $(X,B)$ be a marked space, and let 
$$
F_B=\bigcup_{b\in B} F_b
$$
be the union of disjoint quasiconical closed neighbourhoods of the points in $b$. We also fix the cover $\calU=\{F_B,X\setminus B\}$ of $X$,
and we observe that the  interiors of the elements of $\calU$ still cover the whole of $X$. The marked chain complex $C_*^{\calU,\calM}(F_B,B;\calR)$ is naturally
a subcomplex of  $C_*^{\calU,\calM}(X,B;\calR)$, so we can consider the short exact sequence of complexes
$$
\xymatrix{
C_*^{\calU,\calM}(F_B,B;\calR)\ar[r] & C_*^{\calU,\calM}(X,B;\calR)\ar[r]^-{(\pi_E)_*}  & E_*\ ,
}
$$
where 
$$
E_* = C_*^{\calU,\calM}(X,B;\calR)/C_*^{\calU,\calM}(F_B,B;\calR)\ .
$$
%together with the quotient chain map
%$$
%(\pi_E)_*\colon C_*^{\calU,\calM}(X,B;\calR)\to E_*\ .
%$$
Recall now from Proposition~\ref{PropVanishingHomologyForCones} that the marked homology of the pair $(F_B,B)$ vanishes in every degree. Therefore, by looking at the long exact sequence 
associated to the short exact sequence above we conclude that the quotient map induces an isomorphism 
$$
H_n(\pi_E)\colon H_n^{\calU,\calM}(X,B;\calR)\to H_n(E_*)
$$
for every $n\in\mathbb{N}$. 

The usual relative chain complex $C_*(F_B,B;\calR)$ is naturally a subcomplex of $C_*^\calU(X,B;\calR)$, so we can consider the quotient map
$$
(\pi_D)_*\colon C_*^\calU(X,B;\calR)\to D_*\ =\ C_*^\calU(X,B;\calR)/C_*(F_B,B;\calR)\ .
$$
Since $B$ is a strong deformation retract of $F_B$ we have $H_n(F_B,B;\calR)=0$ for every $n\in\mathbb{N}$, so arguing as above we deduce that the induced map
$$
H_n(\pi_D)\colon H_n^{\calU}(X,B;\calR)\to H_n(D_*)
$$
is an isomorphism for every $n\in\mathbb{N}$.

Observe that the set of (the classes of) singular simplices which are supported in $X\setminus B$ but not entirely contained in $F_B \setminus B$ provides a basis both of $E_*$ and of $D_*$, so
the inclusion $i^\calU_*\colon C_*^{\calU,\calM}(X,B;\calR)\hookrightarrow C_*^\calU(X,B;\calR)$ induces an isomorphism
$$
\alpha_*\colon E_*\to D_*\ .
$$
%We thus have the commutative diagram
%$$
%\xymatrix@C=0.50em{
%H_{n-1}(E_*)\ar[r] \ar[d]^{H_{n-1}(\alpha_*)} & H_n^{\calM}(F_B,B;\calR) \ar[r] \ar[d]^{H_{n}(j_*)} & 
%H_n^{\calU,\calM}(X,B;\calR) \ar[r] \ar[d]^{H_n(i_*^\calU)} & H_{n-1}(E_*)\ar[r] \ar[d]^{H_{n}(\alpha_*)}  & H_{n+1}^\calM(F_B,B;\calR) \ar[d]^{H_{n+1}(j_*)}\\
%H_{n-1}(D_*)\ar[r] & H_n(F_B,B;\calR) \ar[r] & H_n^\calU(X,B;\calR) \ar[r] & H_{n-1}(D_*)\ar[r]  & H_{n+1}(F_B,B;\calR)
%}
%$$
%\todo{Ho rimpicciolito il piu' possibille la larghezza del diagramma, ma forse ora le frecce sono troppo piccole (oltre al fatto che continua ad uscire fuori dal margine)}

Since $H_n(\pi_E)$, $H_n(\pi_D)$, $H_n(\alpha)$ are all isomorphisms, from the commutative diagram
$$
\xymatrix{
H_n^{\calU,\calM}(X,B;\calR) \ar[rr]^-{H_n(\pi_E)} \ar[d]_-{H_n(i_*^\calU)} & & H_{n}(E_*) \ar[d]^-{H_{n}(\alpha)}  \\
 H_n^\calU(X,B;\calR) \ar[rr]_-{H_n(\pi_D)} & & H_{n}(D_*)
}
$$
%Since $\alpha_*$ is an isomorphism already at the level of chains, we have that $H_i(\alpha_*)$ is an isomorphism for every $i\in\mathbb{N}$. Moreover, we know
%from Proposition~\ref{isomF} that also $H_i(j_*)$ is an isomorphism in every degree, so the Five Lemma ensures that the map
we deduce that the map
$$
H_n(i^\calU_*)\colon H_n^{\calU,\calM}(X,B;\calR)\to H_n^\calU(X,B;\calR)
$$
is an isomorphism for every $n\in\mathbb{N}$.

Let us now fix $n\in\mathbb{N}$ and consider the commutative diagram
$$
\xymatrix{
H_n^{\calU,\calM}(X,B;\calR) \ar[r] \ar[d]_-{H_n(i_*^\calU)} & H_n^{\calM}(X,B;\calR)\ar[d]^-{H_n(i_*)}\\
H_n^\calU (X,B;\calR) \ar[r] & H_n(X,B;\calR)\ ,
}
$$
where the horizontal arrows are induced by the inclusions of small chains into generic chains. Theorem~\ref{TeorSmall} and \cite[Prop. 2.21]{hatcher}
ensure that the horizontal arrows are isomorphism, and we have just proved that also the map $H_n(i^\calU_*)$ is an isomorphism.
It follows that the map $H_n(i_*)$ is also an isomorphism, and this concludes the proof of Theorem~\ref{main:basic:thm}.

\section{Ideal simplicial volume}\label{ideal:sec}

In order to define the ideal simplicial volume we need to introduce an $\ell^1$-seminorm on marked homology. 
Henceforth, unless otherwise stated, we will deal only with chains and classes with real coefficients. Therefore, when this does not create ambiguities,
we will omit to specify the coefficients we are working with. The reader will understand that, unless otherwise stated, all the homology (and cohomology) modules 
will have real coefficients.

\subsection{$\ell^1$-(semi)norms and simplicial volumes}\label{seminorms:sub}
Let us
first recall the definition of the $\ell^1$-(semi)norm on ordinary singular homology. If $(X,Y)$ is a topological pair, then for every singular chain
$c=\sum_{i=1}^k a_i\sigma_i \in C_n(X,Y)$ written in reduced form (i.e.~such that $\sigma_i\neq \sigma_j$ if $i\neq j$ and no $\sigma_i$ is supported in $Y\subseteq X$) we set
$$
\|c\|_1=\sum_{i=1}^k |a_i| \ \in \R\ .
$$
%It is readily seen that $(C_n(X,Y), \|\cdot \|_1)$ is a normed vector space. 
The norm $\|\cdot \|_1$ restricts to a norm on the space of relative cycles
$Z_i(X,Y)$, which defines in turn a quotient seminorm (still denoted by $\|\cdot \|_1$) on the homology module $H_i(X,Y)$. If $M$ is a compact oriented $n$-manifold with boundary, then 
$H_n(M,\partial M)\cong \R$ admits a preferred generator, called \emph{real fundamental class}, which is the image via the change of coefficient map of the generator of $H_n(M,\partial M;\mathbb{Z})\cong \mathbb{Z}$ corresponding
to the orientation of $M$. If $[M,\partial M]\in H_n(M,\partial M)$ denotes the real fundamental class of $M$, then the simplicial volume $\|M\|$ of $M$ is defined by
$$
\|M\|=\|[M,\partial M]\|_1\ .
$$

If $(X,B)$ is a marked space, then the $\ell^1$-norm on the relative chain complex $C_*(X,B)$ restricts to an $\ell^1$-norm on $C_*^\calM(X,B)$, which induces in turn a seminorm on
the homology modules $H_*^\calM(X,B)$. We will denote these (semi)norms again with the symbol $\|\cdot \|_1$.

Let now $(M, \partial M)$ be a compact oriented $n$-manifold with boundary, and denote  by $(X,B)$ the associated marked space as defined in Subsection~\ref{associated:marked}, and by $p\colon M\to X$ the natural projection.
Since $\partial M$ is a strong deformation retract of
an open neighboourhood of $\partial M$, the quotient map $p'\colon (M,\partial M)\to (M/\partial M,[\partial M])$ induces an isomorphism
$$
H_k(p')\colon H_k(M,\partial M)\to H_k(M/\partial M,[\partial M])
$$
for every $k\in\mathbb{N}$ (see e.g.~\cite[Prop. 2.22]{hatcher}). For the same reason, also the obvious quotient map
$q\colon(X,B)\to (M/\partial M, [\partial M])$ induces isomorphisms on relative homology in every degree, so from the commutative diagram
$$
\xymatrix{
H_k(M,\partial M) \ar[rr]^{H_k(p')} \ar[rd]^{H_k(p)}& &  H_k(M/\partial M,[\partial M])\\
& H_k(X,B) \ar[ru]^{H_k(q)}& 
}
$$
we can deduce that also the map $H_k(p)\colon H_k(M,\partial M)\to H_k(X,B)$ is an isomorphism in every degreee $k\in\mathbb{N}$. By composing
this map with the inverse of the isomorphism $H_k(i_*)\colon H_k^\calM(X,B)\to H_k(X,B)$ (see Theorem~\ref{main:basic:thm}), we thus obtain, for every $k\in\mathbb{N}$,
the isomorphism
$$
\psi_k\colon H_k(M,\partial M)\to H_k^\calM(X,B)\ .
$$

\begin{defn}\label{ideal:simpl:defn}
The %\emph{integral} (resp.~\emph{real}) 
\emph{ideal fundamental class} of $M$ is the element
%$[M,\partial M]_\matZ^\calM\in H_n^\calM(X,B;\matZ)$ (resp.~
%$[M,\partial M]^\calM \in H_n^\calM(X,B)$
%) defined by
\begin{align*}
%[M,\partial M]_\matZ^\calM&=\psi_n([M,\partial M]_\matZ)\ ,\\
[M,\partial M]^\calM&=\psi_n([M,\partial M])\ \in\ H_n^\calM(X,B)\ ,
\end{align*}
%The \emph{integral ideal simplicial volume} of $M$ is defined by
%$$
%\|M\|_{\calI,\matZ}=\| M,\partial M]_\matZ^\calM\|_1\ ,
%$$
and the \emph{ideal simplicial volume} of $M$ is defined by
$$
\isv{M}=\| [M,\partial M]^\calM\|_1\ .
$$
\end{defn}
%Henceforth, unless otherwise stated, we will deal only with the real ideal simplicial volume. Therefore, when this does not create amibiguities,
%we will omit to specify the coefficients we are working with. The reader will understand that, unless otherwise stated, henceforth all the homology (and cohomology) modules 
%will have real coefficients.

Observe that both the classical and the ideal simplicial volume of an oriented manifold do not depend on its orientation and that it is straightforward
 to extend the definition also to  nonorientable or disconnected manifolds:
if $M$ is connected and nonorientable, then its simplicial volumes are equal
to one half of the corresponding simplicial volumes of its orientable double covering, and the simplicial
volumes of any  manifold is the sum of the corresponding simplicial volumes of its connected components.

Henceforth, every manifold will be assumed to be compact, connected and oriented. 

\begin{rem}\label{local:degree}
Let us describe a characterization of ideal fundamental cycles that will prove useful in the sequel.

Let $(Z,Z')$ be a topological pair such that $Z\setminus Z'$ is an oriented $n$-manifold, and let $\alpha$ be a relative cycle in $C_n(Z,Z')$.
Recall that for every $x\in Z\setminus Z'$ one may define the local degree of $\alpha$ at $x$ as the real number $d(\alpha)(x)$ corresponding
to the element defined by $\alpha$ in $H_n(Z,Z\setminus\{x\})$ via the canonical identification between $H_n(Z,Z\setminus\{x\})$ and $\mathbb{R}$ induced by the orientation
of $Z\setminus Z'$. In fact, the number $d(\alpha)(x)$ is independent of $x\in Z\setminus Z'$.

Let now $M$ be an $n$-dimensional compact oriented manifold with associated marked space $(X,B)$. 
By definition, a marked cycle $z\in C_n^\calM(X,B)$ is an ideal fundamental cycle if and only if, when considered as a cycle
in the relative chain module $C_n(X,B)$, it defines the same homology class as the image via the natural projection of the classical
fundamental class of $M$. Now the fundamental class of $M$ may be characterized as the unique class $\alpha\in H_n(M,\partial M)$ having local 
degree equal to $1$ at every point $p\in M\setminus \partial M$. As a consequence,  $z$ is an ideal fundamental cycle if and only
if it has local degree equal to $1$ at every point $x\in X\setminus B$. 
\end{rem}

\begin{rem}\label{no:naif}
One may wonder why we do not define the ideal simplicial volume just by taking the $\ell^1$-seminorm
of the image of $[M,\partial M]$ in the relative homology module $H_n(X,B)$, where $(X,B)$ is the marked space associated to $M$. 
The following example shows that this choice would not lead to any meaningful invariant.
Let $M$ be the (natural compactification) of the complement of a hyperbolic knot in the $3$-sphere $S^3$. Since the fundamental group of $\partial M$
is abelian, hence amenable, 
we have
$\isv{M}=\|M\|>0$ (see Corollary~\ref{hyperbolic:cor}). 
On the other hand, it is not difficult to show that $X$ is simply connected (indeed, the Wirtinger presentation of $\pi_1(M)$
shows that $\pi_1(M)$ is generated by meridian loops, all of which are killed in $X$). As a consequence, the bounded cohomology
(with real coefficients) of the pair $(X,B)$ vanishes in every positive degree~\cite{Grom82}, and by a standard duality argument this implies in turn that
the $\ell^1$-seminorm of any element in $H_3(X,B)$ vanishes~\cite{Loeh}. 

Another possible definition of ideal simplicial volume would arise from defining admissible simplices as those singular
simplices $\sigma\colon \Delta^n\to X$ such that $\sigma^{-1}(B)$ coincides with the set of vertices of $\Delta^n$. On the one hand,
this choice would lead to ideal fundamental cycles which are closer in spirit to ideal triangulations. On the other hand, with this choice the whole theory
would be much more complicated: for example, showing that admissibly homotopic maps induce the same morphism on marked homology would be much more difficult;
the simplices obtained by subdividing an admissible one would not be admissible; and even the fact that the ordinary simplicial volume bounds from above the
ideal one, if true, would not be obvious at all. However, by Proposition~\ref{verticesinB} the 
ideal simplicial volume can be computed by looking only at ideal fundamental cycles whose simplices have all their vertices in $B$.
\end{rem}

\subsection{(Marked) bounded cohomology and duality}
As anticipated in the introduction, in order to compute the (marked) simplicial volume it is often useful to switch from the study of singular chains to the dual theory of (bounded)
singular cochains. 

Recall that the chain modules $C_i(M,\partial M)$ and $C_i^\calM(X,B)$ are endowed with $\ell^1$-norms, both denoted by the symbol $\|\cdot \|_1$. 
We denote by $C^i_b(M,\partial M)$ (resp.~$C^i_\calM(X,B)$) the topological dual of $C_i(M,\partial M)$ (resp.~of $C_i^\calM(X,B)$), i.e.~the space of bounded linear functionals
on $C_i(M,\partial M)$ (resp.~on $C_i^\calM(X,B)$), endowed with the  operator norm $\|\cdot\|_\infty$ dual to $\|\cdot \|_1$. 

In the case of ordinary (relative) singular (co)chains,
the modules $C^*_b(M,\partial M)$ define a subcomplex of the ordinary singular chain complex $C^i(M,\partial M)$, and are called the \emph{bounded cochains} modules of the pair $(M,\partial M)$.
The $\ell^\infty$-norm of a cochain $\alpha\in C^i_b(M,\partial M)$ is equal to the supremum of the values taken by $\alpha$ on single singular $i$-simplices with values in $M$ (and not supported in $\partial M$).
The cohomology of the complex $C^*_b(M,\partial M)$ is the \emph{bounded cohomology} of the pair $(M,\partial M)$, and it is denoted by $H^*_b(M,\partial M)$. 
The $\ell^\infty$-norm on each $C^i_b(M,\partial M)$ restricts to the subspace of cocycles, and induces a seminorm (still denoted by $\|\cdot\|_\infty$) on $H^i_b(M,\partial M)$,
for every $i\in\mathbb{N}$.

An analogous characterization holds also for the $\ell^\infty$-norm on $C^i_\calM(X,B)$: indeed,
%This norm may be characterized
%as follows: 
for every $\varphi\in C^i_\calM(X,B)$, $$
\|\varphi\|_\infty=\sup \{|\varphi(\sigma)|\, |\, \sigma\ \textrm{admissible singular}\ i\textrm{-simplex}\ \textrm{not supported in} \ B\}\ . 
$$
Just as in the ordinary case, the boundary map $\partial_i\colon C_i^\calM(X,B)\to C_{i-1}^\calM(X,B)$ is bounded with respect to the $\ell^1$-norm, hence
it induces a bounded dual map $\delta^{i-1}\colon C^{i-1}_\calM(X,B)\to C^{i}_\calM(X,B)$, which endows
$(C^*_\calM(X,B),\delta^*)$ with the structure of a normed complex. We denote by
$H^*_\calM(X,B)$ the cohomology of the complex $(C^*_\calM(X,B),\delta^*)$, and we endow each
$H^i_\calM(X,B)$, $i\in\mathbb{N}$, with the seminorm induced by $\|\cdot\|_\infty$, which will still be denoted by $\|\cdot\|_\infty$.

The obvious pairings between $C^i_b(M,\partial M)$ and $C_i(M,\partial M)$ and between $C^i_\calM(X,B)$ and $C_i^\calM(X,B)$ induce pairings
\begin{align*}
 \langle\cdot,\cdot\rangle\colon &H^i_b(M,\partial M)\times H_i(M,\partial M)\to \R\, ,\\
 \langle\cdot,\cdot\rangle\colon &H^i_\calM(X,B)\times H_i^\calM(X,B)\to \R\, .
\end{align*}
The following duality result is proved e.g.~in~\cite{Loeh}, and will be used in the proofs of Theorem~\ref{amenable:thm} and~\ref{lowerbound:thm}.

\begin{prop}\label{duality:prop}
Let $\alpha\in H_i^\calM(X,B)$. Then
$$
\|\alpha\|_1=\max \{\langle\varphi,\alpha\rangle\, |\, \varphi\in H^i_\calM(X,B),\, \|\varphi\|_\infty\leq 1\}\ . 
$$
\end{prop}

\subsection{Ideal simplicial volume vs.~complexity}\label{isv:vs:c}
Recall that a $\Delta$-complex is a topological space obtained by gluing a family of copies of the standard simplex along affine diffeomorphisms of some of their faces.
Therefore, $\Delta$-complexes provide a mild generalization of (geometric realizations of) simplicial complexes, the only differences being that simplices in $\Delta$-complexes need not be embedded (since identifications between pairs of faces of the same simplex are allowed), and that distinct simplices in $\Delta$-complexes may share more
than one face.

\begin{defn}
Let $(M, \partial M)$ be a compact $n$-manifold with boundary with associated marked space $(X,B)$. An \emph{ideal triangulation} of $M$ is a realization of $(X,B)$ as 
a $\Delta$-complex whose set of vertices is equal to $B$. The \emph{complexity} $c(M)$ of $M$ is defined as the minimal number of $n$-dimensional simplices
in any ideal triangulation of $M$.
\end{defn}

Let us now prove that the complexity bounds the ideal simplicial volume from above. Let $M$ be an $n$-manifold with boundary with $c(M)=m$, and let $(X,B)$ be the
marked space associated to $M$. By definition of complexity,
there exists an ideal triangulation of $M$ with $m$ top-dimensional simplices, i.e.~a realization of $X$ as a $\Delta$-complex
with the following properties: $X$ is obtained by gluing $m$ copies $\Delta^n_1,\ldots, \Delta^n_m$ of the standard simplex $\Delta^n$; 
 the set of vertices of the resulting $\Delta$-complex is equal to $B$. Recall that $M$ is oriented, so for every $j=1,\ldots,m$ we can fix an orientation-preserving
 affine identification $\sigma_j\colon \Delta^n\to \Delta^n_j$. It is not quite true that the sum of the $\sigma_j$ defines a marked cycle. Nevertheless, in order to obtain a cycle
 out of the $\sigma_j$ it is sufficient to alternate them as follows. 
 
 Let $\mathfrak{S}_{n+1}$ be the group of permutations of the set $\{0,\ldots,n\}$, and denote by $e_0,\ldots,e_n$ the
 vertices of the standard simplex $\Delta^n$. For every $\tau\in \mathfrak{S}_{n+1}$ we denote
 by $\overline{\tau}$ the unique affine automorphism $\overline{\tau}\colon \Delta^n\to\Delta^n$ such
 that $\overline{\tau}(e_i)=e_{\tau(i)}$ for every $i=0,\ldots,n$. Observe now that for each $j=0,\ldots,m$ and every
 $\tau\in\mathfrak{S}_{n+1}$
 the set $(\sigma_j\circ\overline{\tau})^{-1}(B)$ is equal to the set of vertices of $\Delta^n$,
 so $\sigma_j\circ \overline{\tau}$ is admissible. We may thus define the marked chain
 $$
 z=\frac{1}{(n+1)!}\sum_{j=1}^m \sum_{\tau\in\mathfrak{S}_{n+1}} \varepsilon(\tau)\sigma_j\circ \overline{\tau}\ \in \ C_n^\calM(X,B)\ ,
 $$
 where $\varepsilon(\tau)\in \{1,-1\}$ denotes the sign of $\tau$. 
 
 It is now easy to check that $z$ is indeed a cycle. Moreover, 
 the local degree of $z$ at any point in $X\setminus B$ is equal to one, so Remark~\ref{local:degree} implies that the class $[z]^\calM$ of $z$ in $H_n^\calM(X,B)$ is the ideal fundamental class of $M$. 
 Finally, we have $\|z\|_1= m$, hence
 $$
 \isv{M}=\|[M,\partial M]^\calM\|_1\leq \|z\|_1=m=c(M)\ .
 $$
 This concludes the proof of Theorem~\ref{complexity}.

\subsection{Ideal simplicial volume vs.~classical simplicial volume}\label{class:sub}
Let $\sigma\colon \Delta^i\to M$ be a singular simplex, and recall that $p\colon (M,\partial M)\to (X,B)$ is the 
natural projection.
We say that $\sigma$ is \emph{admissible} if $p\circ \sigma$ is admissible as a simplex with values in the marked space $(X,B)$, i.e.~if
 $\sigma^{-1}(\partial M)$
is a subcomplex of $\Delta^i$. Moreover, a chain $c\in C_*(M,\partial M)$ is admissible if it is (the class of) a linear combination of admissible simplices.
The following lemma provides the key step in the proof of the inequality $\isv{M}\leq \|M\|$, and will be useful also in the proof of Theorem~\ref{bound:degree}.

\begin{lemma}\label{normal:lemma}
Let $M$ be a manifold with boundary,
and let $A'$ be a closed subspace of a perfectly normal topological space $A$. Let also $f\colon (A,A')\to (M,\partial M)$
be a map of pairs. Then $f$ is homotopic (as a map of pairs) to a map $g\colon (A,A')\to (M,\partial M)$ such that
$g^{-1}(\partial M)=A'$.
\end{lemma}
\begin{proof}
A classical result in general topology ensures that $\partial M$ admits a closed collar $\calC$ in $M$. 
More precisely, there exists a subset $\mathcal{C} \subset M$ containing $\partial M$ 
such that the pair $(\calC,\partial M)$ is homeomorphic to the pair
$(\partial M \times [0, 1],\partial M\times\{0\})$. Henceforth we fix an identification
$\calC\cong \partial M \times [0, 1]$ induced by such a homeomorphism.

Since $f$ is a map of pairs, we already know that $f^{-1}(\partial M)\supseteq A'$, and we need to perturb 
$f$ in order to push out of $\partial M$ all points not belonging to $A'$. Since $A$ is perfectly normal, there exists a continuous
map $h\colon A\to [0,1]$ such that $h^{-1}(0)=A'$. Moreover, for every $x\in f^{-1}(\calC)$ we have
$$ f(x)=(m(x),d(x))\in \partial M\times [0,1]\ ,$$
where $m\colon f^{-1}(\calC)\to \partial M$ and $d\colon f^{-1}(\calC)\to [0,1]$ are continuous functions. Let us now define a map
$H\colon A\times [0,1]\to M$ as follows. If $x\in f^{-1}(\overline{M\setminus\calC})$, then $H(x,t)=f(x)$ for every $t\in [0,1]$. If $x\in f^{-1}(\calC)$, then 
$$ H(x,t)=(m(x), \min \{d(x)+th(x), 1\})\ .$$
First observe that, if $x\in f^{-1}(\overline{M\setminus\calC})\cap f^{-1}(\calC)$, then necessarily $f(x)=(m(x),1)$, and this readily implies that
the map $H$ is well defined, hence continuous. Moreover, if $x\in A'$, then $d(x)=h(x)=0$, so $H(x,t)=f(x)\in \partial M$ for every $t\in [0,1]$,
i.e.~the maps $f=H(\cdot, 0)$ and $g=H(\cdot, 1)$ are homotopic as map of pairs. Finally, we have 
$g(x)\in \partial M$ if and only if $f(x)\in \calC$ and $d(x)=h(x)=0$, i.e.~if and only if $x\in A'$, as desired.
\end{proof}

We are now ready to prove that the ideal simplicial volume is  bounded from above by the classical simplicial volume:

\begin{thm}\label{main-in1}
Let $(M, \partial M)$ be an orientable compact manifold with boundary.
Then 
$$
\isv{M}\leq \|M\|\ .
$$
\end{thm}
\begin{proof}
Let $\varepsilon>0$ be given, and let 
$z\in C_n(M,\partial M)$ be a fundamental cycle such that $\|z\|_1\leq \|M\|+\varepsilon$. We would like to modify $z$ into an admissible fundamental cycle without
increasing its $\ell^1$-norm. To this aim,
we briefly discuss a well-known geometric description of cycles in singular homology, also described in~\cite[Section 13.2]{miolibro}.
We refer the reader to~\cite[Section 5.1]{Loeh:IMRN} for an alternative approach to this construction.

Let $z=\sum_{i=1}^k a_i \overline{\sigma}_i$, and assume that $\overline{\sigma}_i\neq \overline{\sigma}_j$ for  $i\neq j$.
One can construct a $\Delta$-complex ${P}$ associated to ${z}$
by gluing $k$ distinct copies $\Delta^n_1,\ldots,\Delta^n_k$ of the standard $n$-simplex $\Delta^n$ as follows.
For every $i$ we fix an identification between $\Delta^n_i$ and $\Delta^n$, so that we may consider
$\sigma_i$ as defined on $\Delta^n_i$.
For every $i=1,\ldots,k$, $j=0,\ldots, n$, we denote by $F^i_j$ the $j$-th face of $\Delta^n_i$,
and by $\partial^i_j\colon \Delta^{n-1}\to F^i_j\subseteq \Delta_i^n$ the usual face inclusion.
We say that the faces $F^i_j$ and $F^{i'}_{j'}$ are \emph{equivalent} if  
$\overline{\sigma}_i|_{F^i_j}=\overline{\sigma}_{i'}|_{F^{i'}_{j'}}$, or, more formally, if $\bb^i_j\circ \overline{\sigma}_i=\bb^{i'}_{j'}\circ\overline{\sigma}_{i'}$. 
We now define a $\Delta$-complex $P$ as follows.
The simplices of $P$ are $\Delta^n_1,\ldots,\Delta^n_k$, and, 
if $F_i^j$, $F_{i'}^{j'}$ are equivalent, then we identify them via the affine diffeomorphism
$\partial_{i'}^{j'}\circ (\partial_i^j)^{-1}\colon F_i^j\to F_{i'}^{j'}$. 
By construction, the maps $\overline{\sigma}_1,\ldots,\overline{\sigma}_k$ glue up to a well-defined continuous map
$f\colon P\to M$. We also define the boundary $\partial P$ of $P$ as the subcomplex of $P$ given by all the $(n-1)$-faces of $P$
which are sent by $f$ entirely into $\partial M$. By definition, the map $f$ is a map of pairs
from $(P,\partial P)$ to $(M,\partial M)$.

For every $i=1,\ldots, k$, let  $\hat{\sigma}_i\colon \Delta^n\to {P}$ be the singular simplex obtained by composing
the identification $\Delta^n\cong \Delta^n_i$ with the quotient map with values in $P$,
and let us consider the singular chain $z_{{P}}=\sum_{i=1}^k a_i \hat{\sigma}_i\in C_n(P)$. By construction, $z_{{P}}$ is a relative cycle in $C_n(P,\partial P)$, and
the push-forward of $z_{{P}}$ via $f$ is equal to ${z}$. Moreover, $\|z_P\|_1=\|z\|_1$.

By applying Lemma~\ref{normal:lemma} to the map $f\colon (P,\partial P)\to (M,\partial M)$ we now obtain a map
$g\colon (P,\partial P)\to (M,\partial M)$ such that $g^{-1}(\partial M)=\partial P$. We now set 
$$
z'=g_*(z_P)\ \in\ C_*(M,\partial M)\ .
$$
Since $g$ is homotopic to $f$ as a map of pairs, the chain $z'$ is a relative fundamental cycle for $M$. Moreover, we have
$$
\|z'\|_1=\|g_*(z_P)\|_1\leq \|z_P\|_1=\|z\|_1\leq \|M\|+\varepsilon\ .
$$
Finally, if $\sigma'_i\colon \Delta^n\to M$ is a singular simplex appearing in $z'$, then
$\sigma^{-1}(\partial M)$ is a (possibly empty) union of $(n-1)$-faces of $\Delta^n$, hence it is admissible. 

Let $p\colon (M,\partial M)\to (X,B)$ be the natural projection of $M$ onto its associated marked space. 
Since $z'$ is admissible, we have
 $p_*(z')\in C^\calM(X,B)$. Moreover, since $z'$ is a relative fundamental cycle for $M$,  the chain $p_*(z')$ is an ideal fundamental cycle for $M$.
 Thus
$$
\isv{M}=\|[p_*(z')]\|_1\leq \|p_*(z')\|_1\leq \|z'\|_1 \leq \|M\|+\varepsilon\ .
$$
Since $\varepsilon$ is arbitrary, this concludes the proof.
\end{proof}

In order to conclude the proof of Theorem~\ref{main:inequalities} we now need to show the following:

\begin{prop}
There exists a constant $K_n>0$ such that
$$
\|M\|\leq K_n\cdot  \isv{M}
$$
for every
 $n$-dimensional manifold $M$.
\end{prop}
\begin{proof}
Let $(X,B)$ be the marked space associated to $M$, let $\pi\colon M\to X$ be the natural projection, and
let $z^\calM=\sum_{i=1}^k a_i\sigma_i\in C_*^\calM(X,B)$ be an ideal fundamental cycle for $M$.
We are going to truncate the singular simplices $\sigma_i$ to obtain a realization of the relative fundamental class of $M$ via singular polytopes.
We will then triangulate these polytopes to get a classical relative fundamental cycle, at the cost of the multiplicative constant $K_n$ mentioned in the statement.

Let us begin with a general definition. Let $Q$ be an $m$-dimensional polytope (i.e.~the convex hull of a finite number of points which span an
$m$-dimensional affine space in some Euclidean space), $m\leq n$. 
We inductively define the notion of \emph{fundamental cycle} for $Q$ as follows. If $m=0$, then $Q$ is a point, and we set $z_Q=\sigma\in C_0(Q)$, where 
$\sigma$ is the constant singular simplex at $Q$. If $0<m\leq n$, we say that a chain $z_Q\in C_m(Q)$ is a fundamental cycle for $Q$ if the following conditions hold:
the class of $z_Q$ in $H_m(Q,\partial Q)$ is indeed a relative fundamental cycle for $(Q,\partial Q)\cong (D^m,\partial D^m)$; if $F_1,\ldots,F_k$ are the facets
of $Q$ (i.e.~the codimension-1 faces of $Q$), then $\partial z_Q=\sum_{i=1}^k c_i$, where $c_i$ is a fundamental cycle for $F_i$ for every $i=1,\ldots,k$.

We now fix $0<\varepsilon<1/(n+1)$ and, for every $0\leq m\leq n$, we define a family $\Omega(m,\varepsilon)$ of oriented polytopes which are obtained by truncating the standard 
simplex $$\Delta^m=\{(t_0,\ldots,t_m)\in \R^{m+1}\, |\, \sum_{i=0}^m t_i=1\}$$
around some of its faces. Namely, we say that a polytope $Q$ belongs to $\Omega(m,\varepsilon)$ if there exists a 
subcomplex $K\subseteq \Delta^m$
 such that
\begin{align*}
Q=\{(t_0,\ldots,t_m)\in \Delta^m\,  |\, & \sum_{i\in A} t_i\leq 1-\varepsilon\ \textrm{whenever}\ A\ \textrm{is\ the\ set\ of\ indices}\\ & \textrm{corresponding\ to\ a\ face\ in}\ K\}\ .
\end{align*}
In other words, $Q$ is obtained by removing from $\Delta^m$ a neighbourhood of $K$ (see Fig.~\ref{Figura1}). We endow $Q$ with the orientation induced by $\Delta^m$, and
every facet of $Q$ will also be oriented as (a subset of) the boundary of $Q$.
The condition $\varepsilon<1/(n+1)$ ensures that $Q$ has nonempty  interior unless $K=\Delta^m$.
If $Q$ lies in $\Omega(m,\varepsilon)$, then
every facet of $Q$  is affinely isomorphic to an element of $\Omega(m-1,\varepsilon)$.

\begin{center}
 \begin{figure}
  \includegraphics{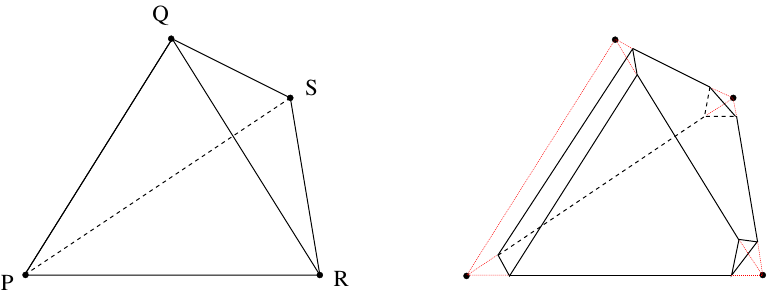}
  \caption{On the left, a standard simplex. On the right, the truncated simplex obtained by removing neighbourhoods of the edge $PQ$ and of the vertices $R$ and $S$.}\label{Figura1}
 \end{figure}

\end{center}

We now claim that for every $0\leq m\leq n$ and for every $Q\in \Omega(m,\varepsilon)$ there exists a fundamental cycle $z_Q$ for $Q$ such that
the following conditions hold:
\begin{enumerate}
 \item $\|z_Q\|_1\leq K_m$, where $K_m$ is a constant only depending on $m$;
 \item let $F_1,\ldots,F_k$ be the facets of $Q$ and let $\partial z_Q=\sum_{i=1}^k c_i$, where $c_i$ is a fundamental cycle for $F_i$ for every $i=1,\ldots,k$.
 If $\varphi\colon F_i\to Q'$ is any orientation-preserving  affine isomorphism
 between $F_i$ and an element $Q'\in\Omega(m-1,\varepsilon)$, then $\varphi_*(c_i)=z_{Q'}$;
 \item if $\varphi\colon Q\to Q$ is an affine isomorphism, then $\varphi_*(z_Q)=\tau(\varphi)z_Q$, where $\tau(\varphi)=\pm 1$ is positive if $\varphi$ is orientation-preserving, and negative
 otherwise.
\end{enumerate}

The existence of the claimed family of fundamental cycles may be proved by induction as follows. If $m=0$, then either $Q=\emptyset$ or $Q$ is a point, and in both cases
the conclusion easily follows. Suppose now that the desired fundamental cycles have been constructed for every $i=0,\ldots,m-1$, and let $Q\in \Omega(m,\varepsilon)$.
For every facet $F$ of $Q$ we have an orientation-preserving affine isomorphism $\varphi\colon F\to Q'$ for some $Q'\in\Omega(m-1,\varepsilon)$, and we set $z_F=\varphi^{-1}_*(z_F)\in C_{m-1}(F)\subseteq C_m(Q)$. 
By (3) (applied to the polytope $Q'$) the chain $z_F$ does not depend on the chosen affine isomorphism. Let us now consider the cycle $z_\partial=\sum_F z_F$. Using (2) and (3) it is immediate to check
that $\partial z_\partial =0$, and this readily implies that $z_\partial$ is a classical fundamental cycle for $\partial Q$. We can then fill $z_\partial$ thus obtaining
a chain $z'_Q\in C_m(Q)$ such that $\partial z'_Q=z_\partial$. If $G$ is the group of the affine isomorphisms of $Q$ into itself, we then set $z_Q=(1/|G|)\sum_{g\in G} \tau(g)g_*(z'_Q)$. Using that
each $z_F$ satisfies (3) we easily see that $\partial z_Q=\partial z'_Q$, and this readily implies that $z_Q$ satisfies (2) and (3). In order to get (1), we only need to observe that
$\Omega(m,\varepsilon)$ is finite, so we may set $K_m=\max \{\|z_Q\|_1\, |\, Q\in\Omega(m,\varepsilon)\}$. Indeed, if $\varepsilon'\neq \varepsilon$, then 
there exists an obvious bijection between $\Omega(m,\varepsilon)$ and $\Omega(m,\varepsilon')$, and fundamental cycles for elements of $\Omega(m,\varepsilon')$ may be chosen to have exactly the same $\ell^1$-norm
as the ones for the corresponding elements of $\Omega(m,\varepsilon)$. This concludes the proof of the claim.

We can now proceed with the proof of the proposition. Let $z^\calM=\sum_{i=1}^s a_i\sigma_i \in C^\calM(X,B)$ be an ideal fundamental cycle for $M$. For every $b\in B$ denote by $F_b$ a closed quasiconical
neighbourhood of $b$ such that $X\setminus \bigsqcup_{b\in B} \inte(F_b)$ is homeomorphic to $M$. We also set $F_B=\bigsqcup_{b\in B} F_b$ and we denote by $\partial F_B$ the topological boundary of $F_B$ in $X$,
so that $(X\setminus \inte(F_B),\partial F_B)$ is homeomorphic to $(M,\partial M)$. 

By continuity, there exists $\varepsilon_0>0$ such that the following condition holds:
for every $i=1,\ldots,s$, 
if $K_i=\sigma^{-1}_i(B)$, then $\sigma_i(\Delta^n\setminus \inte(Q_i))\subseteq \inte(F_B)$, where $Q_i\in \Omega(n,\varepsilon_0)$ is obtained as above by removing from $\Delta^n$  
a neighbourhood of $K_i$. For every $i=1,\ldots,s$ let now $z_{Q_i}$ be the fundamental cycle for $Q_i$ defined above, and set
$$
z_i=(\sigma_i)_*(z_{Q_i})\in C_n(X\setminus B)\, ,\quad z=\sum_{i=1}^s a_iz_i\ .
$$
Using that $z^\calM$ is a marked cycle and the properties of the $z_{Q_i}$ it is not difficult to show that $\partial z\in C_n\left(F_B\setminus B\right)$. In particular,
$z$ is a relative cycle in $C_n\left(X\setminus B,F_B\setminus B\right)$. By retracting each $F_b\setminus \{b\}$ onto $\partial F_b$ we can construct a homotopy equivalence
of pairs 
$$r\colon \left(X\setminus B,F_B\setminus B\right)\to \left(X\setminus \inte(F_B),\partial F_B\right)\cong (M,\partial M)\ .
$$
We claim that the chain $r_*(z)$ is a relative fundamental cycle for $M$. Indeed,  $z^\calM$ is an ideal fundamental cycle, so its image in $H_n(X,B)$
has local degree equal to one at every point $x\in X\setminus B$ (see Remark~\ref{local:degree}). Using this and the fact that every $z_{Q_i}$ is a relative fundamental cycle for $Q_i$,
we deduce that also $z$ has local degree equal to one at every point $x\in X\setminus F_B$. Then, the same holds true for $r_*(z)$, and this 
shows that $r_*(z)$ is a fundamental class for $M$. Therefore,
\begin{align*}
\|M\| &\leq \|r_*(z)\|_1\leq \|z\|_1=\left\|\sum_{i=1}^s a_i(\sigma_i)_*(z_{Q_i})\right\|_1\\ &\leq \sum_{i=1}^s |a_i|\cdot \|z_{Q_i}\|_1\leq
K_n\sum_{i=1}^s |a_i|=K_n\cdot \|z^\calM\|_1
\end{align*}
By taking the infimum over all possible ideal fundamental cycles we get
$$
\|M\|\leq K_n\isv{M}\ ,
$$
which concludes the proof.
\end{proof}

By exploiting the techinques introduced in the proof of Theorem~\ref{main-in1} we may show that the $\ell^1$-norm of a marked homology class
can be computed by looking only at marked chains whose simplices have all their vertices in $B$:

\begin{prop}\label{verticesinB}
Let $(X,B)$ be a marked space such that $H_0(X,B)=0$ (i.e.~$B$ intersects every path connected component of $X$), and let $i\in\mathbb{N}$.
Then, for every $\alpha\in H_i^\calM(X,B)$ and every $\varepsilon>0$ there exists a marked cycle $z\in C_i^\calM(X,B)$ satisfying the following properties:
\begin{enumerate}
\item
$[z]=\alpha$ in $H_i^\calM(X,B)$ and $\|c\|_1\leq \|\alpha\|_1+\varepsilon$;
\item
all the vertices of every singular simplex appearing in $z$ lie in $B$.
\end{enumerate}
\end{prop}
\begin{proof}
Let $z\in C_i^\calM(X,B)$ be an admissible cycle. In order to prove the proposition, 
it is sufficient to show that $z$ is homologous to an admissible cycle $z'\in C_i^\calM(X,B)$ such that $\|z'\|_1\leq \|z\|_1$
and all the vertices of every singular simplex appearing in $z'$ lie in $B$.

%Let us denote by $\mathcal{A}_i$ the set of all the admissible $i$-simplices $\sigma \colon \Delta^i \rightarrow (X, B)$ such that $(\Delta^i)^0 \subseteq \, \sigma^{-1}(B)$, 
%and let  $\calC_n$ the set of all marked $n$-chains whose singular simplices lie in $\calA_n$. The following lemma shows that for the computation of the ideal simplicial volume of $M$ we may restrict to considering ideal fundamental cycles lying in $\calC_n$.

Let $z = \sum_{j = 1}^{k} a_j \sigma_j$.
We exploit the  notation introduced in the proof of Theorem~\ref{main-in1} (with $(M,\partial M)$ replaced by $(X,B)$), 
and we denote by $P$ the $\Delta$-complex associated to $z$ (which is obtained by gluing $k$ copies of the standard simplex $\Delta^i$),
and by $f \colon P \rightarrow X$ the continuous map obtained by gluing the maps $\sigma_1, \cdots, \sigma_k$.
 Moreover, we denote by $\hat{\sigma}_j\colon \Delta^i\to P$ the characteristic map of the $j$-th copy of $\Delta^i$ in $P$
 (see the proof of  Theorem~\ref{main-in1}), and we set
 $z_P = \sum_{j = 1}^{k} a_j \hat{\sigma}_j \in \, C_i(P)$.
 
 By construction, $z_P$ is a relative cycle in $C_i(P, f^{-1}(B))$, and $z = f_*(z_P)$. Moreover,
since every $\sigma_j$ is admissible, the subset $f^{-1}(B)\subseteq P$ is  a subcomplex of $P$. 

% $$z = f_*(z_P) \in \, C_n^{\mathcal{M}}(X, B) \subseteq C_n(X, B).$$ 
We now aim to construct a homotopy of pairs $H \colon (P, f^{-1}(B)) \times I \rightarrow (X, B)$ between $f$ and a map $g$ such that 
every singular simplex appearing in $z' = g_*(z_P)$ is admissible and has all its vertices in $B$. 
We will then have  $[z'] = [z]$ in $H_i(X, B)$, hence in $H_i^{\calM}(X, B)$ thanks to Theorem \ref{main:basic:thm}.
Then the proposition will follow from the inequality $$\|z'\|_1=\|g_*(z_P)\|_1\leq \|z_P\|_1=\|z\|_1\ .$$

Let $v\in P^0\setminus f^{-1}(B)$ be a vertex of $P$ which is not sent to $B$ by $f$. Since $H_0(X,B)=0$, there exists
a continuous path $\gamma_v\colon [0,1]\to X$ joining $v$ with a point in $B$. Moreover, since every point of $B$ has a quasiconical neighbourhood
in $X$, we can safely assume that $\gamma_v^{-1}(B)=\{1\}$, i.e.~the path $\gamma_v$ hits $B$ only at its endpoint.
Let us now consider the homotopy 
$$\widehat{H} \colon (f^{-1}(B) \cup P^0) \times I \rightarrow X\ ,$$ $$\widehat{H}(x, t) = \begin{cases} f(x) &\mbox{ if } x \in \, f^{-1}(B) \\ \gamma_x(t) &\mbox{ if }  x \in \, P^0 \setminus f^{-1}(B)\end{cases}$$
($\widehat{H}$ is indeed continuous since 
both $f^{-1}(B)$ and $P^0 \setminus f^{-1}(B)$ are closed in $P$).
Since $f^{-1}(B)\cup P^0$ is a subcomplex of $P$, thanks to the Homotopy Extension Property for CW-pairs (see \cite[Proposition~0.16]{hatcher})
we can extend $\widehat{H}$ to a homotopy  $H \colon P \rightarrow X$. Moreover, since the homotopy $\widehat{H}$ is relative to $f^{-1}(B)$, we may assume
that $H$ is also relative to $f^{-1}(B)$. In particular, $H$ also defines a homotopy of pairs $H\colon (P,f^{-1}(B))\times I\to (X,B)$.
We then set $g=H(\cdot, 1)$, and we are left to show that every singular simplex appearing in $z' = g_*(z_P)$ is admissible and has all its vertices in $B$.
 
 The fact that every singular simplex appearing in $z' $ has all its vertices in $B$ readily follows from the construction of $g$.
 Moreover, in order to prove that each singular simplex appearing in $z'$ is admissible 
 it is sufficient to show that $g^{-1}(B)$ is a subcomplex of $P$. However,  from the definition of $\widehat{H}$
 and from the explicit description of its extension $H$ given in \cite[Proposition~0.16]{hatcher} we deduce that
 $g^{-1}(B)=f^{-1}(B)\cup P^0$. Since $f^{-1}(B)$ is a subcomplex of $P$, this concludes the proof.
\end{proof}

\subsection{Bounding mapping degrees}\label{bd:deg:sec}
Let us now turn to the proof of Theorem~\ref{bound:degree}. Let $f\colon (M,\partial M)\to (N,\partial N)$ be a map of pairs between
$n$-dimensional manifolds with boundary. We denote by $(X_M,B_M)$ (resp.~$(X_N,B_N)$) the marked space associated to $M$ (resp.~to $N$).

By Lemma~\ref{normal:lemma}, the map $f$ is homotopic (as a map of pair)
to a map $g\colon (M,\partial M)\to (N,\partial N)$ such that $g^{-1}(\partial N)=\partial M$. This condition ensures that
the map $\overline{g}\colon (X_M,B_M)\to (X_N,B_N)$ induced by $g$ is admissible. By looking at the commutative diagram
$$
\xymatrix{
H_n(M,\partial M) \ar[r]^{H_n(g)} \ar[d] & H_n(N,\partial N)\ar[d] \\
H_n(X_M,B_M) \ar[r]^{H_n(\overline{g})} & H_n(X_N,B_N) \\
H_n^\calM(X_M,B_M)\ar[u] \ar[r]^{H_n^\calM(\overline{g})} & H_n^\calM(X_N,B_N)\ar[u]
}
$$
we easily deduce that 
$$
H_n^\calM(\overline{g})([M,\partial M]^\calM)=\deg (g)\cdot [N,\partial N]^\calM= \deg (f)\cdot [N,\partial N]^\calM\ .
$$
Moreover,  the operator
$H_n^\calM(\overline{g})$ is norm non-increasing, since it is induced by a map on marked chains which sends every single admissible simplex to a single admissible simplex.
 Thus
\begin{align*}
\isv{M}&=\|[M,\partial M]^\calM\|_1\geq \|H_n^\calM(\overline{g})([M,\partial M]^\calM)\|_1\\ &=\|\deg (f)\cdot [N,\partial N]^\calM\|_1=|\deg (f)| \cdot \isv{N}\ .
\end{align*}
This concludes the proof of Theorem~\ref{bound:degree}.

\section{The universal covering of a marked space}\label{universal:subsec}
When computing the (co)homology of a space, it is often useful to work with (co)invariant (co)chains on coverings. In
order to implement this strategy in the context of marked (co)homology, we first need to define coverings of marked spaces.
For the sake of simplicity (and since this will be sufficient for our purposes) we only define the universal covering of a marked space, even though our construction
may be easily adapted to define a more general notion of covering between marked spaces. 

Let $(X,B)$ be a marked space, and assume for simplicity that $Y=X\setminus B$ admits a universal covering $q\colon \widetilde{Y}\to Y$. Then we define
the universal covering $(\widetilde{X},\widetilde{B})$ of $(X,B)$ as follows. Let $\{F_b\}_{b\in B}$ be a collection of disjoint quasiconical closed neighbourhoods
of the points of $B$. 
We pick a set $\widetilde{B}$ endowed with a fixed bijection with the set of connected components of $q^{-1}(\bigcup_{b\in B} (F_b\setminus \{b\}))$ and we define 
$\widetilde{X}=\widetilde{Y}\cup \widetilde{B}$. Observe that the covering projection $q\colon \widetilde{Y}\to Y$ extends to a map $\pi\colon \widetilde{X}\to X$
sending to $b\in B$ all the points of $\widetilde{B}$ corresponding to a 
component of $q^{-1}(F_b)$.

In order to turn the pair $(\widetilde{X},\widetilde{B})$ into a marked space, we endow $\widetilde{X}$ with the unique topology such that the following conditions hold:
the subset $\widetilde{Y}=\widetilde{X}\setminus \widetilde{B}$ is open in $\widetilde{X}$ and inherits from $\widetilde{X}$ the topology of $\widetilde{Y}$ as total space of the universal covering of $Y$; 
if $\widetilde{b}\in \widetilde{B}$ corresponds to a component $\widetilde{F}$ of $q^{-1}(F_b\setminus \{b\})$, then a basis of neighbourhoods of $\widetilde{b}$
in $\widetilde{X}$ is given by the collection $\{\widetilde{F}\cap \pi^{-1}(U)\}$, as $U$ varies in a basis of neighbourhoods of $b$ in $X$.

\begin{lemma}\label{univ:marked:lemma}
The pair $(\widetilde{X},\widetilde{B})$ is indeed a marked space, and the map $\pi\colon (\widetilde{X},\widetilde{B})\to (X,B)$ is admissible.
\end{lemma}
\begin{proof}
Let us check that every $\widetilde{b}\in\widetilde{B}$ admits a quasiconical neighbourhood in $\widetilde{X}$. Let $b=\pi(\widetilde{b})$ and let $F_b$ be a quasiconical
neighbourhood of $b$ in $X$. Let also $H\colon F_b\times [0,1]\to F_b$ be a contracting homotopy as in Definition~\ref{coarse:cone}, and denote by 
$$H'\colon \left(F_b\setminus \{b\}\right)\times [0,1)\to F_b\setminus \{b\}
$$
the restriction of $H'$ to $\left(F_b\setminus \{b\}\right)\times [0,1)$ (which is well defined thanks to the properties of $H$). Let also $\widetilde{F}\subseteq \widetilde{X}\setminus \widetilde{B}$ be the connected
component of $q^{-1}(F_b\setminus\{b\})$ corresponding to $\widetilde{b}$.
Covering theory ensures that the homotopy $H'$ lifts to a homotopy 
$$
\widetilde{H}'\colon \widetilde{F}\times [0,1)\to \widetilde{F}
$$
such that $\widetilde{H}'(x,0)=x$ for every $x\in\widetilde{F}$. We now set $\widetilde{F}_{\widetilde{b}}=\widetilde{F}\cup \{\widetilde{b}\}$, and we extend $\widetilde{H}'$
to a homotopy 
$$
\widetilde{H}\colon \widetilde{F}_{\widetilde{b}}\times [0,1]\to \widetilde{F}_{\widetilde{b}}
$$
by setting $\widetilde{H}(x,t)=\widetilde{b}$ whenever $x=\widetilde{b}$ or $t=1$. 
It is now easy to check that the map $\widetilde{H}$ is continuous and satisfies the requirements of Definition~\ref{coarse:cone}, hence $\widetilde{F}_{\widetilde{b}}$ is a closed
quasiconical neuighbourhood of $\widetilde{b}$ in $\widetilde{X}$. As $\widetilde{b}$ varies in $\widetilde{B}$, the construction just described provides disjoint
sets, and this concludes the proof that $(\widetilde{X},\widetilde{B})$ is a marked space.

The fact that $\pi^{-1}(B)=\widetilde{B}$ is obvious, while the continuity of $\pi$ readily follows from the definitions. Thus $\pi$ is admissible, whence the conclusion.
\end{proof}

%It is an easy exercise to check that $(\widetilde{X},\widetilde{B})$ is indeed a marked space (but see Remark~\ref{topology:on:cones:rem}), and that 
%the map $\pi\colon (\widetilde{X},\widetilde{B})\to (X,B)$ is admissible.
 In general, the map $\pi\colon \widetilde{X}\to X$ is no longer a covering (in fact,
it is a covering if and only if the image  of $\pi_1(F_b\setminus\{b\})$ into $\pi_1(X\setminus B)$ 
is trivial for every $b\in B$).

%Nevertheless,
%we will call the map $\pi\colon (\widetilde{X},\widetilde{B})\to (X,B)$ the \emph{universal covering} of the marked space $(X,B)$. In fact,
%we will see in Section~\ref{ideal:sec} that to every manifold with boundary $M$ there is associated a marked space $(X_M,B_M)$. With our definitions,
%the universal covering of $(X_M,B_M)$ will correspond to the marked space associated to the universal covering of $M$ (see Remark~\ref{universal:associated:rem}).

\begin{rem}\label{universal:associated:rem}
Let $M$ be a compact manifold with boundary with associated marked space $(X,B)$. Let  $f\colon \widetilde{M}\to M$ be the universal covering of $M$, and let
%Also denote by 
%$p_M\colon (M,\partial M)\to (X,B)$ and 
$p_{\widetilde{M}}\colon (\widetilde{M},\partial \widetilde{M})\to (X_{\widetilde{M}},B_{\widetilde{M}})$
be the natural projection on the associated marked space. It is very natural to ask whether $(X_{\widetilde{M}},B_{\widetilde{M}} )$ may be identified with the universal covering
$(\widetilde{X},\widetilde{B})$ of $(X,B)$.

First observe that the identification 
$$\widetilde{M}\setminus \partial \widetilde{M}= \widetilde{M\setminus \partial M}\cong
\widetilde{X\setminus B}$$
extends to a projection ${p'}_{\widetilde{M}}\colon \widetilde{M}\to \widetilde{X}$ which is continuous and 
sends every connected component of $\partial \widetilde{M}$ to a single point in $\widetilde{B}$. As a consequence, there exists a (unique) map
$\alpha\colon X_{\widetilde{M}}\to \widetilde{X}$ such that the following diagram commutes:
% map
%$p'_{\widetilde{M}}$ factors through $\widetilde{X}_1$ as follows:
%makes the following diagram commute:
$$
\xymatrix{
\widetilde{M} \ar[rr]^-{{p_{\widetilde{M}}}} \ar[dr]_-{{p'}_{\widetilde{M}}} & & X_{\widetilde{M}} \ar[ld]^-\alpha\\
& \widetilde{X}\ . &
}
$$
Since $X_{\widetilde{M}} $ is endowed with the quotient topology, the map $\alpha$ is continuous. Moreover, it is easy to check that it is bijective. However,
we now show that $\alpha$ is \emph{not} a homeomorphism if $\partial \widetilde{M}$ contains a non-compact connected component $\partial_0\widetilde{M}$. 

Let
 $\widetilde{b}=p_{\widetilde{M}}(\partial_0\widetilde{M})\in B_{\widetilde{M}}$ 
be the point corresponding to $\partial_0\widetilde{M}$. We claim that $\widetilde{b}$ does not admit a countable basis of neighbourhoods in $X_{\widetilde{M}}$.
In fact, 
let $\{U_n\}_{n\in\mathbb{N}}$ be a countable collection of neighbourhoods of $\widetilde{b}$ in $X_{\widetilde{M}}$. Let $d$ be a distance inducing the topology of $\widetilde{M}$ 
(every topological manifold is metrizable), and choose a diverging sequence $\{x_n\}_{n\in\mathbb{N}}$ in $\partial_0\widetilde{M}$. 
For every $n\in\mathbb{N}$ there exists $\varepsilon_n>0$ such that the $d$-ball $B(x_n,\varepsilon_n)$ is contained in $p_{\widetilde{M}}^{-1}(U_n)$.
Using that the sequence $\{x_n\}_{n\in\mathbb{N}}$ diverges, it is not difficult to construct an open neighbourhood $V$ of $\partial_0\widetilde{M}$ in $\widetilde{M}$ such that 
$V$ does not contain $B(x_n,\varepsilon_n)$ for every $n\in\mathbb{N}$. The projection $p_{\widetilde{M}}(V)$ of $V$ is now an open neighbourhood
of $\widetilde{b}$ in $X_{\widetilde{M}}$ which does not contain any of the $U_n$, $n\in\mathbb{N}$. This shows that  $\widetilde{b}$ does not admit a countable basis of neighbourhoods in $X_{\widetilde{M}}$.

In order to show that $\alpha$ is not a homeomorphism it is now sufficient to observe that $\alpha(\widetilde{b})$ has a countable basis of neighbourhoods in $\widetilde{X}$.
In fact,
let $b=\pi(\alpha(\widetilde{b}))$ be the projection of $\alpha(\widetilde{b})$ under the covering of marked spaces $\pi\colon \widetilde{X}\to X$
introduced above. It follows from the very definition of the topology of $\widetilde{X}$ that a basis of neighbourhoods at $\alpha(\widetilde{b})$ is given by the set $\{\pi^{-1}(U)\}$ as $U$ varies
in a basis of neighbourhoods of $b$ in $X$. Now it is easy to check that $X$ is second countable (for example, it is metrizable and compact), hence $b$ has a countable
basis of neighbourhoods in $X$. This implies in turn that $\alpha(\widetilde{b})$ has a countable
basis of neighbourhoods in $\widetilde{X}$, thus concluding the proof that $\alpha$ is not a homeomorphism.

This fact may seem a bit annoying at first sight, but
our choice for the definition of the topology of $\widetilde{X}$ 
 allows us to lift admissible simplices from $X$ to its covering (see Lemma~\ref{lift:lemma} and Remark~\ref{topology:on:cones:rem}).
\end{rem}

\subsection{Lifting admissible simplices}
Let $Y=X\setminus B$ be as before, and
let $\Gamma$ denote the automorphism group of the universal covering $q\colon \widetilde{Y}\to Y$. It is immediate to check that 
every element of $\Gamma$ extends to  an admissible map of $(\widetilde{X},\widetilde{B})$ into itself.
As it is customary for ordinary (co)homology, we would like to compute the marked (co)homology of $X$ by looking at (co)invariant (co)chains on the universal covering.
To this aim we need the following important: 

\begin{lemma}\label{lift:lemma}
Let $\pi\colon \widetilde{X}\to X$ be the universal covering constructed above, and let
$\sigma\colon \Delta^i\to X$ be an admissible singular simplex. Then, there exists an admissible singular simplex $\widetilde{\sigma}\colon \Delta^i\to \widetilde{X}$ such that
$\pi\circ\widetilde{\sigma}=\sigma$. Moreover, if $\widetilde{\sigma}'\colon \Delta^i\to\widetilde{X}$ is any admissible singular simplex
such that  $\pi\circ\widetilde{\sigma}'=\pi\circ\widetilde{\sigma}=\sigma$, then there exists an element $\gamma\in\Gamma$ such that 
$\widetilde{\sigma}'=\gamma\circ \widetilde{\sigma}$.
\end{lemma}
\begin{proof}
%Let $\sigma\colon \Delta^i\to X$ be an admissible singular simplex.
If $\sigma(\Delta^i)\subseteq B$, then by discreteness of $B$ we have $\sigma(\Delta^i)=\{b\}$ for some $b\in B$. We then choose an element
$\widetilde{b}$ in $\pi^{-1}(b)$ and define $\widetilde{\sigma}(x)=\widetilde{b}$ for every $x\in\Delta^i$. The conclusion follows from the fact that 
$\Gamma$ acts transitively on the fiber of $b$.

We may thus set $K=\sigma^{-1}(B)$ and 
suppose that $A=\Delta^i\setminus K$ is non-empty. Since $\sigma$ is admissible, the set $K$ is a subcomplex of $\Delta^i$, and $A$ is  a  convex subset of $\Delta^i$.
In particular, $A$ is simply connected. Since the restriction of ${\pi}$ to $\widetilde{X}\setminus \widetilde{B}$ is a classical covering, 
this implies that there exists a continuous lift $\beta\colon A\to \widetilde{X}$ of the restriction $\sigma|_A$. Let us fix as usual a family $F_b$, $b\in B$, of disjoint closed quasiconical
neighbourhoods of the points of $B$, and let $\widetilde{F}_{\widetilde{b}}$, $\widetilde{b}\in\widetilde{B}$, be the family of disjoint closed quasiconical
neighbourhoods of  points of $\widetilde{B}$ obtained by lifting the $F_b$ (see the proof of Lemma~\ref{univ:marked:lemma}).

Let $K_0$ be a connected component of $K$. Since $B$ is discrete, there exists $b\in B$ such that $\sigma (K_{0})=\{b\}$.
Since $\sigma$ is admissible and non-constant,
the subset $A$ contains $\Delta^i\setminus \partial\Delta^i$. In particular, every point $x\in K_0$ is an accumulation point for $A$. 
By continuity of $\sigma$, 
there exists an open neighbourhood $U$ of $K_0$ in $\Delta^i$ such that $U\cap (K\setminus K_0)=\emptyset$ and $\sigma(U)\subseteq F_b$. Moreover, we can choose $U$ so that
$U\setminus K_0=U\cap A$ is path connected. We now have 
$$
\beta(U\setminus K_0)\subseteq\pi^{-1}(\sigma(U\setminus K_0))\subseteq \pi^{-1}(F_b)=\bigsqcup_{\widetilde{b}\in\pi^{-1}(b)} \widetilde{F}_{\widetilde{b}}\ .
$$
Since $U\setminus K_0$ is connected, this shows that
there exists a unique $\widetilde{b}\in \pi^{-1}(B)$ such that $\beta(U\setminus K_0)\subseteq \widetilde{F}_{\widetilde{b}}$, and we set
$\widetilde{\sigma}(x)=\widetilde{b}$ for every $x\in K_0$. Using the definition of the topology of $\widetilde{X}$ it is not difficult to show that
$\widetilde{\sigma}\colon \Delta^i\to\widetilde{X}$
is continuous. By construction, it is also admissible, thus providing the desired lift of $\sigma$. 

Suppose now
that $\widetilde{\sigma}'\colon \Delta^i\to \widetilde{X}$ is an arbitrary lift of $\sigma$. The maps $\widetilde{\sigma}'|_A$ and
$\widetilde{\sigma}|_A$ both lift $\sigma|_A$, so it readily follows from classical covering theory that $\widetilde{\sigma}'|_A=\gamma\circ 
\widetilde{\sigma}|_A$ for some $\gamma\in\Gamma$. 
Observe now for every $b\in B$, $x\in X\setminus\{b\}$ there exist disjoint neighbourhoods of $b$ and $x$ in $X$. Therefore,
since $A$ is dense in $\Delta^i$ the singular simplices $\widetilde{\sigma}'$ and $\gamma\circ 
\widetilde{\sigma}$ coincide on the  whole of $\Delta^i$, and this concludes the proof.
\end{proof}

\begin{rem}\label{topology:on:cones:rem}
Let $M$ be a compact manifold with universal covering $\widetilde{M}$, and denote by $X_M$, $X_{\widetilde{M}}$ the marked spaces associated
to $M$ and $\widetilde{M}$, respectively. As observed in Remark~\ref{universal:associated:rem}, there exists a natural bijection between $\widetilde{X}$ and $X_{\widetilde{M}}$,
so it makes sense to ask whether Lemma~\ref{lift:lemma} would hold with $\widetilde{X}$ replaced by $X_{\widetilde{M}}$. The answer is negative in general.

For example, let $M=S^1\times [0,1]$, so that $\widetilde{M}=\R\times [0,1]$, and let $\sigma\colon \Delta^1\to \widetilde{M}$ be defined by
$\sigma(t)=(1/t,t)$ if $t\neq 0$ and $\sigma(0)=(0,0)$. Let also $p_{\widetilde{M}}\colon \widetilde{M}\to X_{\widetilde{M}}$ be the natural projection,
and $q\colon X_{\widetilde{M}}\to X_M$ the projection obtained by precomposing the map $\widetilde{X}\to X_M$ with the bijection
$X_{\widetilde{M}}\cong \widetilde{X}$. Then it is not difficult to prove that the simplex
$q\circ p_{\widetilde{M}}\circ \sigma \colon \Delta^1\to X$ is continuous, hence admissible. On the contrary, the map $p_{\widetilde{M}}\circ\sigma\colon \Delta^1\to X_{\widetilde{M}}$
is \emph{not} continuous: in fact, if $U=\{(t,y)\in\R\times [0,1],\, t>1\ \textrm{and}\ y<1/t\}\subseteq \widetilde{M}$, 
then the set $p_{\widetilde{M}}(U)$ is open in $X_{\widetilde{M}}$ and intersects the image of $p_{\widetilde{M}}\circ\sigma$ only at
the point $p_{\widetilde{M}}(\R\times \{0\})$, against the continuity of $p_{\widetilde{M}}\circ\sigma$. 
(On the other hand, the simplex $p_{\widetilde{M}}\circ\sigma\colon \Delta^1\to X_{\widetilde{M}}$ would be continuous if we endowed the set $X_{\widetilde{M}}$ with the topology inherited from the bijection
$X_{\widetilde{M}}\cong \widetilde{X}$.) Using this fact, one can easily show that the admissible simplex $q\circ p_{\widetilde{M}}\circ \sigma$ does not lift
to an admissible simplex with values in $X_{\widetilde{M}}$.
 \end{rem}

\section{Manifolds with amenable boundary}\label{amenable:sec}

Let $M$ be a manifold with boundary with associated marked space $(X,B)$. As usual, we assume that $M$ is compact, connected and oriented.
Denote by $\widetilde{M}$ the universal covering of $M$, and let us fix an identification of the fundamental group
of $M$ with the group $\Gamma$ of  automorphisms of the covering $\pi\colon \widetilde{M}\to M$.
Let $(X,B)$ be the marked space associated to $M$, and 
denote
by $(\widetilde{X},\widetilde{B})$ the universal covering of $(X,B)$ as a marked space.
% or, equivalently, the
%marked space associated to $\widetilde{M}$. The natural projection $(M,\partial M)\to (X,B)$ restricts
%to a  homeomorphism between $M\setminus \partial M$ and $X\setminus B$, and $\pi_1({M})=\pi_1(M\setminus \partial M)$, so
%$\Gamma$ acts on $\widetilde{X}$ via admissible automorphisms (see Subsection~\ref{universal:subsec}).

\subsection{Proof of Theorem~\ref{amenable:thm}}
Let us now suppose that the fundamental group of every component of $\partial M$ is amenable. We would like to prove that
$\isv{M}=\|M\|$. We already know that $\isv{M}\leq \|M\|$, so we need to show the converse inequality $\|M\|\leq \isv{M}$. 
The duality principle for ordinary singular (co)homology 
(see e.g.~\cite{Loeh} or~\cite{miolibro})
implies that there exists a bounded cohomology class $\psi\in H^n_b(M,\partial M)$ such that 
$\|\psi\|_\infty\leq 1$ and $\langle \psi,[M,\partial M]\rangle =\|M\|$.

Let $\varepsilon>0$ be given.
Since the fundamental group of each component of $\partial M$ is amenable, by~\cite[Corollary 5.18]{miolibro} there exists a representative $f\in C^n_b(M,\partial M)$ of $\psi$ such that 
$\|f\|_\infty\leq \|\psi\|_\infty +\varepsilon$
which is \emph{special} in the following sense.
\begin{itemize}
\item
 Let ${\sigma},{\sigma}'\colon \Delta^n\to M$ be singular simplices which lift to
 maps  $\widetilde{\sigma},\widetilde{\sigma}'\colon \Delta^n\to \widetilde{M}$ such that,
for every $i=0,\ldots,n$, at least one of the following conditions holds:
 either $\widetilde{\sigma}(e_i)=\widetilde{\sigma}'(e_i)$, or $\widetilde{\sigma}(e_i)\neq\widetilde{\sigma}'(e_i)$ but ${\widetilde{\sigma}}(e_i)$ and $\widetilde{\sigma}'(e_i)$
belong to the same connected component of $\partial\widetilde{M}$,
where $e_0,\ldots,e_n$ are the vertices of the standard
$n$-simplex. Then $f({\sigma})=f({\sigma}')$.
\end{itemize}

We are now ready to construct a cocycle $f^\calM\in C^n_\calM(X,B)$ out of the cocycle $f$. 
Let us denote by $p'_{\widetilde{M}}\colon \widetilde{M}\to \widetilde{X}$ the projection described in 
Remark~\ref{universal:associated:rem}.

Let $\sigma\colon \Delta^n\to X$ be an admissible simplex, and let $\widetilde{\sigma}\colon \Delta^n\to \widetilde{X}$ be a lift
of $\sigma$, as described in Lemma~\ref{lift:lemma}. It may be the case that $\widetilde{\sigma}$ is not the projection on $\widetilde{X}$
of any singular simplex with values in  $\widetilde{M}$.
Nevertheless, we can arbitrarily choose a singular simplex $\widehat{\sigma}\colon \Delta^n\to \widetilde{M}$ such that
$\widetilde{\sigma}$ and $p'_{\widetilde{M}}\circ\widehat{\sigma}$ coincide on the vertices of $\Delta^n$.
We then set
$$
f^\calM(\sigma)=f(\pi \circ \widehat{\sigma})\ ,
$$
where $\pi \colon \widetilde{M} \rightarrow M$ is the universal covering projection.

The fact that $f$ is special implies that $f^\calM$ is well defined. Moreover,
it is readily seen that $f^\calM$ is a cocycle, and
$$
\|f^\calM\|_\infty\leq \|f\|_\infty\leq \|\psi\|_\infty +\varepsilon\leq 1+\varepsilon\ .
$$

Let $\psi^M\in H^n_\calM(X,B)$ be the class represented by $f^\calM$. We are now going to evaluate
$\psi^M$ on the ideal fundamental class of $M$. So let
$z\in C_n(M,\partial M)$ be a fundamental cycle for $M$. The proof of Theorem~\ref{main-in1} shows that
we can modify $z$ into an admissible fundamental cycle without
increasing its $\ell^1$-norm. 
Therefore, we may assume that the chain $z$ is admissible. 
As a consequence, the chain $z^\calM=p_*(z)\in C_n^\calM(X,B)$ is also admissible. By definition, it follows that
$z^\calM$ is an ideal fundamental cycle for $M$. Moreover, it readily follows from the definition of $f^\calM$ that
$f^\calM(z^\calM)=f(z)$.
We then have
$$
\langle \psi^\calM,[M,\partial M]^\calM\rangle=f^\calM(z^\calM)=f(z)=\langle \psi,[M,\partial M]\rangle=\|M\|\ .
$$
Moreover, 
$$
\|\psi^\calM\|_\infty\leq \|f^\calM\|_\infty\leq 1+\varepsilon\ .
$$
Therefore, by applying Proposition~\ref{duality:prop} with $\alpha=[M,\partial M]^\calM$ and $\varphi=\psi^\calM/(1+\varepsilon)$ we get 
$$
\isv{M}=\|[M,\partial M]^\calM\|_1\geq \frac{\langle \psi^M,[M,\partial M]^\calM\rangle}{1+\varepsilon}=\frac{\|M\|}{1+\varepsilon}\ .
$$
Since $\varepsilon$ is arbitrary, we thus have
$$
\isv{M}\geq \|M\|\ .
$$
This concludes the proof of Theorem~\ref{amenable:thm}.

\section{Hyperbolic manifolds with geodesic boundary}\label{hyperbolic:sec}
This section is devoted to the proof of Theorem~\ref{geodesic:boundary:thm}, which computes the ideal simplicial volume of  an infinite family of hyperbolic $3$-manifolds
with geodesic boundary. Mimicking the well-known computation of the classical simplicial volume for hyperbolic manifolds due to Gromov and Thurston,
we will first establish the lower bound on the ideal simplicial volume in terms of the Riemannian volume described in Theorem~\ref{lowerbound:thm}.
To this aim, rather than defining a proper straightening of simplices as in the classical case, we will directly exploit  the duality principle described in Proposition~\ref{duality:prop}.
In the $3$-dimensional case, building on a recent result on 
 the volumes of peculiar classes of (partially) truncated tetrahedra in hyperbolic space,
we will then show that the resulting lower bound is sharp for an infinite family of $3$-manifolds. 
%by exhibiting 
% ideal triangulations whose number of tetrahedra realize the bound.

\subsection{A geometric realization of the associated marked space}\label{geom:real}
We denote by $\calH^n$ the hyperboloid model of hyperbolic space $\matH^n$, i.e.~we set
$$
\calH^n=\{x\in \R^{n+1}\, |\, \langle x,x\rangle=-1,\, x_0>0\}\ ,
$$
where $$ \langle (x_0,\ldots,x_n),(y_0,\ldots,y_n)\rangle=-x_0y_0+\sum_{i=1}^n x_iy_i$$
is the Minkowsky scalar product. We also denote by $\calS^n$ the hyperboloid
$$
\calS^n=\{x\in \R^{n+1}\, |\, \langle x,x\rangle=1\}\ .
$$
For every element $x\in \calS^n$ we define the dual hyperplane $H(x)$ (resp.~dual half-space $H^+(x)$) of $\calH^n$ by setting
$$
H^+(x)=\{y\in \calH^n\, |\, \langle y,x\rangle \leq 0\}\, , $$
$$
H(x)=\partial H^+(x)=\{y\in \calH^n\, |\, \langle y,x\rangle = 0\}\ .
$$ 

Let now $M$ be a compact hyperbolic manifold with non-empty geodesic boundary. The universal covering $\widetilde{M}$ 
of $M$ is a convex subset of $\matH^n$ bounded by infinitely
many disjoint hyperplanes, and the fundamental group $\Gamma$ of $M$ acts by isometries on $\widetilde{M}$. In fact, every isometry of $\widetilde{M}$ is the restriction
of a unique element in $SO(n,1)$, so $\Gamma$ acts also on $\calS^n$. By taking the dual vectors of the boundary component of $\widetilde{M}$, one can construct
a $\Gamma$-invariant countable set $\widehat{D}\subseteq \calS^n$ such that %the following conditions hold: 
%\begin{itemize}
%\item
for every
$\hat{q}\in \widehat{D}$ the dual hyperplane $H(q)\subseteq \calH^n$
is a connected component of $\partial \widetilde{M}$, and
%\item
%we have
$$
\widetilde{M}=\bigcap_{\hat{q}\in \widehat{D}} H^+(\hat{q})\ .
$$
%\end{itemize}
%Let $\widehat{Y}$ denote the closed convex hull of $\widehat{D}$ in $\R^{n+1}$.
%The following facts are proved in~\cite{Kojima:survey} (see also~\cite{FP}):
Moreover, up to acting via an isometry of $\calH^n$, we may suppose that $(1,0,\ldots,0)$ belongs to $\inte(\widetilde{M})=\widetilde{M}\setminus\partial \widetilde{M}$, and this easily implies
that the set $\widehat{D}$ is contained in the half space $\{x_0>0\}$.
%\begin{enumerate}
%\item
%There exists an affine space-like hyperplane of $\R^{n+1}$ separating $\widehat{Y}$ from the origin;
%\item
%the set $\widehat{Y}$ is a polyhedron of $\R^{n+1}$ having $\widehat{D}$ as set of vertices (in particular, the boundary of $\widehat{Y}$ coincides with the union
%of its $n$-dimensional faces). 
%\end{enumerate}
%By (1), possibly after the action of an isometry of $\calH^n$, 
%we may suppose that $\widehat{Y}$ is contained in the half-space $\{x_0>0\}$. 
This allows us to fruitfully exploit the projective model of hyperbolic space: indeed,
let $U\cong \R^n$ be the affine chart defined by $U=\{[x]\in \mathbb{P}^n(\R)\, |\, x_0\neq 0\}$, consider the projection
 $\pi\colon \mathbb{R}^{n+1}\setminus\{0\}\to \mathbb{P}^n(\R)$ 
and let $P\subseteq U$ be the projective model of hyperbolic space, i.e.~set $P=\pi(\calH^n)$. Also set 
${D}=\pi(\widehat{D})\subseteq U$. With a slight abuse, we denote simply by $\widetilde{M}$
the set $\pi(\widetilde{M})\subseteq P$ and, for every $\hat{q}\in\widehat{D}$, we denote by $H^+(\hat{q})$ also the projection
$\pi(H^+(\hat{q}))\cap U\subseteq U$.
%and for every $q\in D$ let $C(q)\subseteq U$ be the cone over $q$ of the connected component
%of $\partial \widetilde{M}$ corresponding to $q$ (where convex combinations are taken with respect to the affine structure of $U$). Then
%$$
%Y=\widetilde{M}\cup \bigcup_{q\in \overline{D}} C(q)\ .
%$$

Let now $(X,B)$ be the marked space associated to $M$, and let $(\widetilde{X},\widetilde{B})$ be the universal covering of the marked space $(X,B)$.
%or, equivalently, the marked space associated to $\widetilde{M}$.
Let us also set $Y=\inte(\widetilde{M})\cup D$. 
We will now fix the the following identification between $(\widetilde{X},\widetilde{B})$ and $(Y,D)$: recall from Remark~\ref{universal:associated:rem}
that there exists a natural bijection between $(\widetilde{X},\widetilde{B})$ and $(X_{\widetilde{M}},B_{\widetilde{M}})$; this induces in turn 
an identification between $\widetilde{X}\setminus \widetilde{B}$ and $X_{\widetilde{M}}\setminus B_{\widetilde{M}}\cong \inte(\widetilde{M})$; moreover,
$\widetilde{B}\cong B_{\widetilde{M}}$ admits a natural bijection with the set of connected components of $\widetilde{M}$, hence with $D$.
We will endow the pair $(Y,D)$ with the structure of marked space inherited from this identification. Observe that the identification
$(\widetilde{X},\widetilde{B})\cong (Y,D)$ is equivariant with respect to the action of $\G$ on $(Y,D)$.

\subsection{(Partially) truncated simplices}
%Let $H$ be a $k$-dimensional affine subspace of $U\subseteq \mathbb{P}^n(\R)$. Then the intersection
%$H\cap P$, if non-empty, is a totally geodesic 
%$k$-dimensional subspace of hyperbolic space (and, in fact, every totally geodesic subspace of $P$ can be obtained in this way).
%As a consequence, if $\Delta\subseteq Y$ is the support of a straight $n$-dimensional simplex, then 
%the set $\Delta\cap P$, if non-empty, 
%is a totally geodesic polyhedron of $P$ with a finite number of faces. The volume of $\Delta\cap P\neq \emptyset$ is infinite if and only
% $\Delta$ is not contained in a hyperplane of $U$ and at least one vertex of $\Delta$ does not belong to $\overline{P}$. If this is the case (and if some other technical
%conditions are satisfied),
%then by truncating the infinite volume ends of $\Delta\cap P$ one obtain a so-called \emph{(partially) truncated simplex}. More precisely,
%we have the following:
Just as ideal simplices are the fundamental building blocks for cusped hyperbolic manifolds, truncated simplices may be exploited to construct hyperbolic manifolds 
with geodesic boundary. Just as the name suggests, truncated simplices are obtained by truncating Euclidean simplices in the chart $U\subseteq \mathbb{P}^n(\R)$
along the dual hyperplanes of their hyperideal vertices. Here is a precise definition:

\begin{defn}
Let $v_0,\ldots,v_{n}$ be points of $(U\setminus \partial P) \subseteq \mathbb{P}^n(\R)$, let $\Delta\subseteq U$ be
the convex hull of the $v_i$, and let $I\subseteq \{0,\ldots,n\}$ be the set of indices $i$ such that
$v_i\notin \overline{P}$. Also suppose that the following conditions hold:
\begin{enumerate}
%\item The points $v_0,\ldots,v_n$ are not all contained in a hyperplane of $U$;
\item For every $i,j=0,\ldots,n$, $i\neq j$, the straight segment joining $v_i$ with $v_j$ intersects $P$;
\item For every $i\in I$, let $\hat{v}_i\in \calS^n$ be the lift of $v_i$ 
such that $x_0(v_i)>0$. Then ${v}_j\in H^+(\hat{v}_i)$ for every $j\notin I$.
\end{enumerate}
Then the set
$$
\Delta^*=\Delta \cap \left(\bigcap_{i\in I} H^+(\hat{v}_i)\right)
$$
is a \emph{partially truncated simplex} (see Figure~\ref{truncated:fig}). 

The $v_i$ are the \emph{vertices} of $\Delta^*$. A vertex $v_i$ is \emph{finite} if  $v_i\in P$, and
\emph{ultraideal} if $v_i\in U\setminus \overline{P}$. The simplex $\Delta^*$ is \emph{degenerate} if its vertices are all contained in a hyperplane of $U$.

The intersection
of $\Delta^*$ with any face of $\Delta$ is an \emph{internal} face of $\Delta^*$. We say that an internal edge of $\Delta^*$ is \emph{fully hyperideal}
if both the vertices of the corresponding edge of $\Delta$ are hyperideal, and that $\Delta^*$ is fully truncated if every vertex of $\Delta$ is hyperideal.
For every $i\in I$, the set $\Delta^*\cap H^+(\hat{v}_i)$ is 
the \emph{truncation simplex} of $\Delta^*$. When $n=3$, truncation simplices
are usually called truncation triangles. 
The dihedral angle between a truncation simplex and any internal face adjacent to it is equal to $\pi/2$.  
(Partially) truncated tetrahedra are compact; in particular, they have finite volume.
%The volume of 
%$\Delta^*$ is always finite.
%Moreover, if no $v_i$ is ideal, then $\Delta^*$ is compact.

%Let $\Delta\subseteq Y$ be the support of a straight $n$-dimensional simplex, and assume that the following conditions hold:
%\begin{enumerate}
%\item 
%All the vertices of $\Delta$ are contained in $\widetilde{M}\cup D\subseteq Y$;
%\item
%the intersection $\Delta^*=\Delta\cap \widetilde{M}$ is non-empty.
%\end{enumerate}
%(When (1) holds, condition (2) is equivalent to the fact that there does not exist a point $q\in D$ such that all the vertices of $\Delta$
%are equal to $q$.)

%Then the geodesic polyhedron $\Delta^*\subseteq \widehat{M}\subseteq P$ is a \emph{partially truncated simplex}. 

%The intersection
%of $\Delta^*$ with any face of $\Delta$ is an \emph{internal} face of $\Delta^*$. If $q\in D$ is a vertex of $\Delta$ and $H\subseteq \widetilde{M}$ is the 
%boundary component corresponding to $q$, then $\Delta\cap H$ is a \emph{truncation simplex} of $\Delta^*$. When $n=3$, truncation simplices
%are usually called truncation triangles. 

%The dihedral angle between a truncation simplex and any internal face adjacent to it is equal to $\pi/2$.  
\end{defn}

\begin{rem}
 In our definition of (partially) truncated tetrahedra we did not allow ideal vertices, nor (partially) ideal truncation simplices
 (which occur when the straight segment joining two hyperideal vertices is tangent to $P$ at a point in $\partial P$). Truncated tetrahedra with ideal vertices
 arise in the decomposition of hyperbolic $n$-manifolds with geodesic boundary and rank-$(n-1)$ cusps, while truncated tetrahedra 
 with ideal truncation simplices arise in the decomposition of hyperbolic manifolds with non-compact geodesic boundary. This choice allows us to avoid some technicalities, 
 and the study of compact truncated simplices is sufficient for our applications.
  \end{rem}

\begin{figure}
\begin{center}
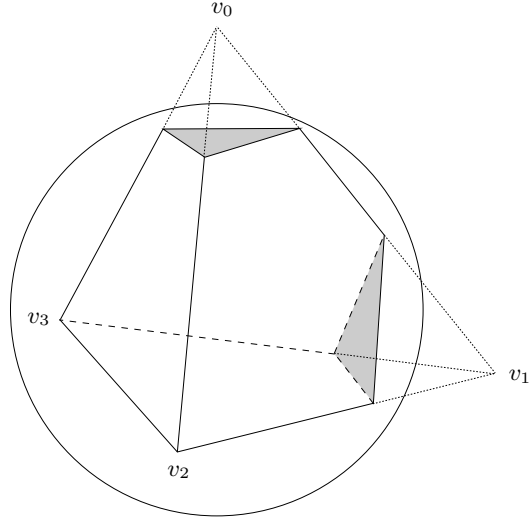
\caption{A partially truncated tetrahedron. The vertices $v_0$ and $v_1$ are hyperideal, while $v_2$ and $v_3$ are finite.}\label{truncated:fig}
\end{center}
\end{figure}

%We now specialize our discussion to the case $n=3$. 
For any geodesic segment $e$ in $P$ we denote by $L(e)$ the hyperbolic length of $e$.

\begin{defn}\label{defvl}
Let $\ell>0$. Then we set for any $n \in \mathbb{N} $
\begin{align*}
V_\ell^n=\sup \{  \vol (\Delta^*)\, |\, & \Delta^*\ \textrm{fully\ truncated } n-\mbox{simplex}\, , \\ & L(e)\geq \ell \ \textrm{for\ every\ internal\ edge\ of}\ \Delta^*\}\ .
\end{align*}
\end{defn}

\subsection{The marked volume form}
Let $M$ be a compact hyperbolic manifold with geodesic boundary. 
Recall from the introduction that $\ell(M)$ denotes the smallest return length of $M$,
i.e.~the length of the shortest
path with both endpoints on $\partial M$ which intersects $\partial M$ orthogonally at each of its endpoints (equivalently, 
it is the smallest distance between distinct boundary components of the universal covering of $M$).
Let $(X,B)$ be the marked space associated to $M$, and let $(Y,D)\cong (\widetilde{X},\widetilde{B})$  be the marked space
described in Subsection~\ref{geom:real}. We are going to define a marked volume form on $M$ by assigning to every admissible simplex $\sigma$
the signed volume of a partially truncated simplex having the same vertices as a lift of $\sigma$.

To this aim, let us first define the function
$$
\algvol\colon Y^{n+1}\to \R
$$
as follows. Take an element $(y_0,\ldots,y_n)\in Y^{n+1}$. 
If the $y_i$ are all contained in a hyperplane of $P$ (e.g.~if $y_i=y_j$ for some $i\neq j$), then we simply set $\algvol(y_0,\ldots,y_n)=0$.
Otherwise, it is easy to check that $\{y_0,\ldots,y_n\}$ is the set of vertices of a non-degenerate (partially) truncated simplex $\Delta^*\subseteq \widetilde{M}$, and we set
$$
\algvol(y_0,\ldots,y_n)=\varepsilon(y_0,\ldots,y_n)\cdot \vol(\Delta^*)\ ,
$$
where $\varepsilon(y_0,\ldots,y_n)=1$ (resp.~$\varepsilon(y_0,\ldots,y_n)=-1$) 
if the barycentric parametrization $$\Delta^n\to U\, ,\qquad (t_0,\ldots,t_n)\mapsto t_0y_0+\ldots t_ny_n $$of the Euclidean simplex 
with vertices $v_0,\ldots,v_n$ is orientation-preserving (resp.~orientation-reversing). Here we are endowing $U$ with the unique orientation which induces
the positive orientation on $\widetilde{M}\subseteq P\subseteq U$.

Observe now that there is an obvious isometric isomorphism between $C^n_\calM (X,B)$ and $C^n_\calM (Y,D)^\Gamma$.
We can then define the cochain $\omega\in C^n_\calM(Y,D)^\Gamma$ by setting
$$
\omega(\widetilde{\sigma})=\algvol(v_0,\ldots,v_n)\ ,
$$
where $v_i=\widetilde{\sigma}(e_i)$ and $e_i$ is the $i$-th vertex of the standard simplex $\Delta^n$. It is esay to check that
$\omega$ is a cocycle, hence we may denote by $[\omega]\in H^n_\calM(X,B)$ the coclass represented by $\omega$.

\begin{prop}\label{omegaclass}
 We have
 $$
 \langle [\omega],[M,\partial M]^\calM\rangle =\vol(M)\ .
 $$
 \end{prop}
\begin{proof}
It is sufficient to exhibit 
an ideal fundamental cycle $z_M\in C_n^\calM(X,B)$ for which  $\omega(z_M)=\vol(M)$.

Kojima's canonical decomposition for hyperbolic manifolds with geodesic boundary~\cite{Kojima, Kojima2} shows that
$M$ can be decomposed into fully truncated polyhedra. Unfortunately, it is not clear whether these polyhedra  may be coherently subdivided to give
a triangulation of $M$ by non-degenerate truncated simplices. However, a decomposition of $M$ into possibly degenerate truncated simplices may be obtained by 
subdividing  Kojima's decomposition
and, if needed, by inserting a finite number of degenerate truncated simplices between the faces of the original polyhedra. One can finally obtain the desired
fundamental cycle $z_M$ by alternating the barycentric parametrizations of these simplices just as we did in Subsection~\ref{isv:vs:c}.
It is then obvious that $\omega(z_M)=\vol(M)$, and this concludes the proof.

%(As an alternative, one can exploit a triangulation by fully truncated simplices of a degree-$d$ covering $M'$ of $M$ (which exists by~\cite{LST}) to construct
%an ideal fundamental cycle $z_{M'}$ for $M'$ supported on fully truncated simplices. The desired cycle $z_M$ can then be defined as $1/d$ times the push-forward
%of $z_{M'}$ via the map induced by the covering $M'\to M$ on the associated marked spaces.

%Now it follows from the definitions that $\omega(z_M)=\vol(M)$, whence the conclusion.
\end{proof}

\subsection{Proof of Theorem~\ref{lowerbound:thm}}
We are now ready to conclude the proof of the lower bound on $\isv{M}$ in terms of the Riemannian volume of $M$. 
%First observe that every internal edge of any fully truncated simplex with vertices in $Y=\inte(\widetilde{M})\cup D$
%has length not smaller than $\ell(M)$.
By Proposition~\ref{verticesinB}, for every $\epsilon>0$ we may find an ideal fundamental cycle $$z_M=\sum_{i=1}^k a_i\sigma_i \ \in\ C_n^\calM(X,B)$$
such that all the vertices of every $\sigma_i$ lie in $B$, and $\|z_M\|_1\leq \isv{M}+\varepsilon$. As a consequence, since
the length of any internal edge of any fully truncated simplex with vertices in $Y=\inte(\widetilde{M})\cup D$
is not smaller than $\ell(M)$, for every $i=0,\ldots, k$ we have $|\omega(\sigma_i)|\leq V_{\ell(M)}^{n}$. Therefore,

\begin{align*}
\vol(M) &=\omega(z_M)\leq \sum_{i=0}^k |a_i|\cdot |\omega(\sigma_i)|\leq V_{\ell(M)}^{n}\sum_{i=0}^k |a_i|=V_{\ell(M)}^{n}\|z_M\|_1\\ 
& \leq V_{\ell(M)}^{n} \left(\isv{M}+\varepsilon \right)\ .
\end{align*}

Since $\varepsilon$ is arbitrary, this concludes the proof of Theorem~\ref{lowerbound:thm}.

\subsection{The class $\calM_g$}

As in the introduction, let $\calM_g$, $g\geq 2$, be the class of
3-manifolds $M$ with boundary that admit an ideal triangulation by $g$ tetrahedra
and have Euler characteristic $\chi(M)=1-g$ (so $\chi(\bb M)=2-2g$).
We also denote by $\overline\calM_g$ the set
of hyperbolic $3$-manifolds $M$ with connected geodesic boundary such that 
$\chi(\bb M)=2-2g$. It turns out that elements of $\calM_g$ admit a geometric decomposition into \emph{regular}
truncated tetrahedra.

A fully truncated tetrahedron $\Delta$ is regular if any permutation of its vertices is realized by an isometry of the truncated tetrahdron.
Equivalently, $\Delta$ is regular if and only if the dihedral angles along its edges are all equal to each other, and this happens
if and only if the hyperbolic lengths of its internal edges are all equal to each other. 
Up to isometry, regular truncated tetrahedra are parametrized by their edge length (which may vary in $(0,+\infty)$, or by their dihedral angles (which may vary in $(0,\pi/3)$).
It is well known (see for example~\cite{Ush}) that, if $\theta(\ell)$ denotes the dihedral
angles along the internal edges of a regular truncated tetrahedron with edge lengths equal to $\ell$, then the map
$\ell\to\theta(\ell)$ is strictly increasing, and has limits
$$ \lim_{\ell \to 0^+} \theta(\ell)=0$$
(when the truncated tetrahedron is tending to a regular ideal hyperbolic octhedron), and
$$\lim_{\ell\to +\infty} \theta(\ell)=\pi/3$$
(when the truncated tetrahedron is tending to a regular ideal hyperbolic tetrahedron).
%We refer the reader to Appendix~\ref{appendix} for more details on the geometry and the trigonometry of truncated tetrahedra.

%\todo{Togliere l'ultima riga? Ho messo una citazione si Ushijima perche' c'e' anche il grafico (inverso). Mentre non so se in generale mettere le hyperreferenze all'appendice si faccia o meno. Se sei contrario, la tolgo}

For every $g\geq 2$, we denote by $\ell_g$ the length of the internal edges of the regular truncated tetrahedron
with dihedral angles equal to $\pi/(3g)$. Moreover, for every $\ell>0$ we denote by $\Delta_\ell$ the (isometry class of the)
regular truncated tetrahedron with edge length equal to $\ell$.
It is proved e.g.~in~\cite{KM} that 
%the volume of the regular truncated tetrahedron with dihedral angles all equal to $\theta$ is equal to
%$$
%v_8-3\int_0^{\theta}
%{\rm arccosh}\, \left(\frac{\cos t}{2\cos t-1}\right)\, {\rm d}t \ ,
%$$
%where $v_8$ is the volume of the regular ideal octahedron. In particular, the volume of regular truncated tetrahedra
%is monotonically decreasing with respect to the dihedral angles (hence, also to the lengths) of the internal edges. 
%We also have 
$$
\vol(\Delta_{\ell_g})=%8L\left(\frac{\pi}{4}\right)
v_8-3\int_0^{\frac{\pi}{3g}}
{\rm arccosh}\, \left(\frac{\cos t}{2\cos t-1}\right)\, {\rm d}t \ ,
$$
where $v_8$ is the volume of the regular ideal octahedron.
%\label{voldg}\\
%& =%4G
%v_8-3\int_0^{\frac{\pi}{3g}}
%{\rm arccosh}\, \left(\frac{\cos t}{2\cos t-1}\right)\, {\rm d}t\ , \notag
%\end{align}

The following result is a restatement of the main theorem of~\cite{FrMo2}, and shows that, at least for every $\ell\leq \ell_2$, the 
regular tetrahedron $\Delta_\ell$ has the largest volume
among all fully truncated tetrahedra whose edge lengths are not smaller that $\ell$:

\begin{thm}\label{maxvol}
 Let $\ell\leq \ell_2$. Then
 $$
 V_\ell^3=\vol(\Delta_\ell)\ .
 $$
\end{thm}

The following result lists some known properties of manifolds belonging to $\calM_g$.
The last point implies that $\calM_g$ coincides with the set of the elements
of $\overline\calM_g$ of smallest volume.

\begin{prop}[\cite{FriMaPe,KM}]\label{Mg:prop}
Let $g\geq 2$. Then:
\begin{enumerate}
\item
the set $\calM_g$ is nonempty;
\item 
every element of $\calM_g$ admits a hyperbolic structure with geodesic boundary
(which is unique up to isometry by Mostow Rigidity Theorem);
\item
the boundary of every element of $\calM_g$ is connected, so $\calM_g\subseteq \overline\calM_g$; 
\item
if $M\in\calM_g$, then $M$ decomposes into the union of $g$ copies of $\Delta_{\ell_g}$, so
in particular $\vol(M)=g\vol(\Delta_{\ell_g})$;
\item
if $M\in\overline\calM_g$,
then $\vol(M)\geq g\vol(\Delta_{\ell_g})$;
\item
if $M\in\calM_g$, then
the shortest return length $\ell(M)$ of $M$ is equal to $\ell_g$;
\item
if $M\in\overline\calM_g$, then
$\ell(M)\geq \ell_g$.
\end{enumerate}
\end{prop}
\begin{proof}
Items (2) and (3) and (4) are proved in~\cite{FriMaPe}, 
items~(1) and (5) in~\cite{KM,M}. Item~(7) is proved in~\cite[Lemma 5.3]{M}, and (6) follows from (7) together with the fact that
the internal egdes of the truncated tetrahedra in the decomposition described in~(4) define return paths of length $\ell_g$.
\end{proof}

 \subsection{Proof of Theorem~\ref{geodesic:boundary:thm} and Corollary~\ref{minimal:isv}}
 We are now ready to prove Theorem~\ref{geodesic:boundary:thm}. Let 
 $M\in\overline{\calM}_g$. By Proposition~\ref{Mg:prop} we have
 $\vol(M)=g\cdot \vol(\Delta_{\ell_g})$ and $\ell(M)=\ell_g$. Moreover, Theorem~\ref{maxvol} ensures that
 $V_{\ell_g}^3=\vol(\Delta_{\ell_g})$, so plugging these equalities into the lower bound given by Theorem~\ref{lowerbound:thm}
 we obtain
 $$
 \isv{M}\geq \frac{\vol(M)}{V_{\ell(M)}^3}=\frac{g\cdot \vol(\Delta_{\ell_g})}{\vol(\Delta_{\ell_g})}=g\ .
 $$
 On the other hand, the manifold $M$ admits an ideal triangulation with $g$ tetrahedra, so
 $$
 \isv{M}\leq c(M)\leq g\ .
 $$
 We thus get
 $$
 \isv{M}=c(M)=g\ ,
 $$
 which proves the first statement of Theorem~\ref{geodesic:boundary:thm}.
 
 If $M\in\overline{\calM}_g$, Proposition~\ref{Mg:prop} implies that $\vol(M)\geq g\cdot \vol(\Delta_{\ell_g})$ and $\ell(M)\geq \ell_g$, so
 that $V_{\ell(M)}^3\leq V_{\ell_g}^3=\vol(\Delta_g)$.
 Therefore, we have 
 $$
 \isv{M}\geq \frac{\vol(M)}{V_{\ell(M)}^3}\geq\frac{g\cdot \vol(\Delta_{\ell_g})}{\vol(\Delta_{\ell_g})}=g\ .
 $$
This concludes the proof of Theorem~\ref{geodesic:boundary:thm}.

In order to prove Corollary~\ref{minimal:isv}, it is sufficient to observe that, if $M$ is an oriented compact hyperbolic $3$-manifold with geodesic boundary having an ideal simplicial
volume not greater than $2$, then Corollary~\ref{cor:intro:lower} ensures that
$$
2\geq \isv{M}\geq \frac{\vol(M)}{v_8}\geq \vol(\partial M)\ ,
$$
where the last inequality is due to Miyamoto~\cite[Theorem 4.2]{M}. In particular, from Gauss-Bonnet Theorem we deduce that $\partial M$ is a connected surface of genus $2$,
so $\isv{M}\geq 2$ by Theorem~\ref{geodesic:boundary:thm}. We have thus proved that the elements of $\calM_{2}$ are exactly the compact hyperbolic $3$-manifolds
with geodesic boundary having the smallest possible ideal simplicial volume.

\subsection{An application to mapping degrees}\label{degrees:sub}
Take elements $M\in\calM_g$ and $M'\in\calM_{g'}$, where $g\geq g'$.
As stated in Corollary~\ref{degree:cor}, Theorems~\ref{bound:degree} and~\ref{geodesic:boundary:thm} imply that any map of pairs
$$
f\colon (M,\partial M)\to (M',\partial M')
$$
satisfies the inequality
\begin{equation}\label{boundgg}
\deg(f)\leq \frac{g}{g'}\ .
\end{equation}
The following proposition implies that there are cases where this bound is sharp:

\begin{prop}
Let $g'\geq 2$ and let $g=k\cdot g'$, $k\in\mathbb{N}\setminus \{0\}$, be a multiple of $g'$.
For every $M'\in \calM_{g'}$ there exist a manifold $M\in\calM_g$ and a map of pairs $f\colon (M,\partial M)\to (M',\partial M')$ such that
$$
\deg(f)=\frac{g}{g'}\ .
$$
\end{prop}
 \begin{proof}
 As proved e.g.~in~\cite{FriMaPe}, manifolds in $\calM_{g'}$ can be characterized as those orientable manifolds that admit an ideal triangulation with $g'$ tetrahedra and
 only one edge,
 meaning that all the edges of the truncated tetrahedra of the triangulation are identified to each other in the manifold.
 
 Let us denote by $e'$ the unique edge of the triangulation of $M'$. Then by removing from $M'$ a small open neighbourhood of $e'$ one gets a genus--$(g'+1)$
 handlebody. In particular, a small loop encircling $e'$ has infinite order in $\pi_1(M'\setminus e')$, so for every $k\geq 1$ we can consider the cyclic ramified covering
 $f\colon M\to M'$ having order $k$ and ramification locus equal to $e'$. By construction, the ideal triangulation of $M'$ lifts to an ideal triangulation of $M$ having $g=k\cdot g'$
 tetrahedra and only one edge $e=f^{-1}(e')$. Thus $M$ belongs to $\calM_g$. Since the map $f$ has topological degree equal to $k$,  this concludes the proof.
 \end{proof}

 One could bound the degrees of maps between manifolds with boundary also by looking at the ordinary simplicial volume both of the boundaries and of the doubles of the manifolds
 involved. Indeed, if $f\colon (M,\partial M)\to (M',\partial M')$ is a map of pairs between oriented manifolds with boundary of the same dimension, then 
 $f$ restricts to a map $g\colon \partial M\to\partial M'$ and extends to a map $F\colon DM\to DM'$ between the double $DM$ of $M$ and the double $DM'$ of $M'$.
 Moreover,
 $$
 \deg (F)=\deg (g)=\deg (f)\ .
 $$
 Therefore, by exploiting the usual bounds of mapping degrees in terms of the ordinary simplicial volume, one gets
 \begin{equation}\label{double:est}
 \deg (f)=\deg (F)\leq \frac{\|DM\|}{\|DM'\|}\ ,
 \end{equation}
  \begin{equation}\label{boundary:est}
 \deg (f)=\deg (g)\leq \frac{\|\partial M\|}{\|\partial M'\|}\ .
 \end{equation}
 Let us show that, at least when $M\in\calM_g$ and $M'\in\calM_{g'}$, these bounds are less effective that the bound~\eqref{boundgg} obtained by exploiting
 the ideal simplicial volume. Indeed, in this case the inequality~\eqref{double:est} ensures that, for every map of pairs
 $f\colon (M,\partial M)\to (M',\partial M')$, one has
 $$
 \deg(f)\leq \frac{\|DM\|}{\|DM'\|}=\frac{\vol(DM)}{\vol(DM')}=\frac{\vol(M)}{\vol(M')}=\frac{g\vol(\Delta_{\ell_g})}{g'\vol(\Delta_{\ell_{g'}})}\ ,
 $$
 and the right-hand side of this inequality is strictly bigger than $g/g'$, since $g > g'$ implies $\ell_g<\ell_{g'}$ and $\vol(\Delta_{\ell_g})>\vol(\Delta_{\ell_{g'}})$. 
 For example, if $g'=2$ and $g=2k$ is very big, then $\vol(\Delta_{\ell_{g'}})\approx 3.226$, while $\vol(\Delta_{\ell_{g}})\approx v_8=3.664$, so the integral part
 of  $(g\vol(\Delta_{\ell_g}))/(g'\vol(\Delta_{\ell_{g'}}))\approx 1.135 (g/g')$ is strictly bigger than the (sharp) bound $g/g'$.

 On the other hand, in this case the inequality~\eqref{boundary:est} gives
 $$
 \deg(f)\leq \frac{\|\partial M\|}{\|\partial M'\|}=\frac{|\chi(\partial M)|}{|\chi(\partial M')|}=\frac{g-1}{g'-1}\ ,
 $$
 and again the right hand side of this inequality is strictly bigger than $g/g'$. For example, if $g'=2$ this bound gives $\deg(f)\leq g-1$,
 which is a much less restrictive condition than the inequality $\deg(f)\leq g/2$ provided by the study of ideal simplicial volume. 
 
 \bibliographystyle{amsalpha}
 \bibliography{biblionote}

\newcommand{\etalchar}[1]{$^{#1}$}
\providecommand{\bysame}{\leavevmode\hbox to3em{\hrulefill}\thinspace}
\providecommand{\MR}{\relax\ifhmode\unskip\space\fi MR }
% \MRhref is called by the amsart/book/proc definition of \MR.
\providecommand{\MRhref}[2]{%
  \href{http://www.ams.org/mathscinet-getitem?mr=#1}{#2}
}
\providecommand{\href}[2]{#2}
\begin{thebibliography}{BBF{\etalchar{+}}14}

\bibitem[Ada87]{Adams}
C.~C. Adams, \emph{The noncompact hyperbolic 3-manifold of minimal volume},
  Proc. Amer. Math. Soc. \textbf{100} (1987), 601--606.

\bibitem[BBF{\etalchar{+}}14]{BBFIPP}
M.~Bucher, M.~Burger, R.~Frigerio, A.~Iozzi, C.~Pagliantini, and M.~B.
  Pozzetti, \emph{Isometric properties of relative bounded cohomology}, J.
  Topol. Anal. \textbf{6} (2014), no.~1, 1--25.

\bibitem[Bel16]{Bele16}
I.~Belegradek, \emph{Topology of open nonpositively curved manifolds.},
  Geometry, Topology, and Dynamics in Negative Curvature (London Mathematical
  Society Lecture Note Series, pp. 32-83), vol. 425, Cambridge: Cambridge
  University Press, 2016.

\bibitem[BFP15]{BFP}
M.~Bucher, R.~Frigerio, and C.~Pagliantini, \emph{The simplicial volume of
  3-manifolds with boundary}, J. Topol. \textbf{8} (2015), 457--475.

\bibitem[BFP17]{BFP2}
M.~Bucher, R.~Frigerio, and C.~Pagliantini, \emph{A quantitative version of a
  theorem by {J}ungreis}, Geom. Dedicata \textbf{187} (2017), 199--218.

\bibitem[BGS85]{BGS85}
W.~Ballmann, M.~Gromov, and V.~Schroeder, \emph{Manifolds of nonpositive
  curvature}, Progress in Mathematics, vol.~61, Birk{\"a}user Boston Inc.,
  1985.

\bibitem[FFM12]{FFM}
S.~Francaviglia, R.~Frigerio, and B.~Martelli, \emph{Stable complexity and
  simplicial volume of manifolds}, J. Topol. \textbf{5} (2012), 977--1010.

\bibitem[FM]{FrMo2}
R.~Frigerio and M.~Moraschini, \emph{On volumes of truncated tetrahedra with
  constrained edge lengths}, to appear in Period. Math. Hungar., DOI:
  10.1007/s10998-018-00277-8.

\bibitem[FMP03]{FriMaPe}
R.~Frigerio, B.~Martelli, and C.~Petronio, \emph{Complexity and {H}eegaard
  genus of an infinite class of compact 3-manifolds}, Pacific J. Math
  \textbf{210} (2003), 283--297.

\bibitem[Fri17]{miolibro}
R.~Frigerio, \emph{Bounded cohomology of discrete groups}, Mathematical Surveys
  and Monographs, vol. 227, Americal Mathematical Society, 2017.

\bibitem[Gro82]{Grom82}
M.~Gromov, \emph{Volume and bounded cohomology}, Publ. Math. Inst. Hautes
  \'Etudes Sci. \textbf{56} (1982), 5--99.

\bibitem[Hat02]{hatcher}
A.~Hatcher, \emph{Algebraic topology}, Cambridge University Pres, Cambridge,
  2002.

\bibitem[KK15]{KK}
S.~Kim and T.~Kuessner, \emph{Simplicial volume of compact manifolds with
  amenable boundary}, J. Topol. Anal. \textbf{7} (2015), 23--46.

\bibitem[KM91]{KM}
S.~Kojima and Y.~Miyamoto, \emph{The smallest hyperbolic 3-manifolds with
  totally geodesic boundary}, J. Differential Geom. \textbf{34} (1991),
  175--192.

\bibitem[Koj90]{Kojima}
S.~Kojima, \emph{Polyhedral decomposition of hyperbolic manifolds with
  boundary}, Proc. Work. Pure Math. \textbf{10} (1990), 37--57.

\bibitem[Koj92]{Kojima2}
\bysame, \emph{Polyhedral decomposition of hyperbolic 3-manifolds with totally
  geodesic boundary}, Aspects of Low-Dimensional Manifolds, Adv. Studies Pure
  Math., vol.~20, 1992, pp.~93--112.

\bibitem[L{\"o}h07]{Lothesis}
C.~L{\"o}h, \emph{Homology and simplicial volume}, Ph.D. thesis, WWU
  M{\"u}nster, 2007, available online at
  http://nbn-resolving.de/urn:nbn:de:hbz:6-37549578216.

\bibitem[L{\"o}h08]{Loeh}
\bysame, \emph{Isomorphisms in $l^1$-homology}, M{\"u}nster J. Math. \textbf{1}
  (2008), 237--266.

\bibitem[L{\"o}h16]{Loeh:IMRN}
\bysame, \emph{Finite functorial semi-norms and representability}, Int. Math.
  Res. Not. IMRN \textbf{2016} (2016), no.~12, 3616--3638.

\bibitem[LS09]{Loh-Sauer}
C.~L{\"o}h and R.~Sauer, \emph{Degree theorems and {L}ipschitz simplicial
  volume for non-positively curved manifolds of finite volume}, J. Topol.
  \textbf{2} (2009), 193--225.

\bibitem[Miy94]{M}
Y.~Miyamoto, \emph{Volumes of hyperbolic manifolds with geodesic boundary},
  Topology \textbf{33} (1994), 613--629.

\bibitem[Ush06]{Ush}
A.~Ushijima, \emph{A volume formula for generalised hyperbolic tetrahedra},
  Non-Euclidean geometries. Mathematics and Its Applications \textbf{531}
  (2006), 249--265.

\end{thebibliography}
\end{document}